\documentclass[namedate,webpdf,imanum]{ima-authoring-template}
\newcommand{\R}{\mathbb{R}}
\newcommand{\C}{\mathbb{C}}
\newcommand{\Z}{\mathbb{Z}}
\newcommand{\vx}{\mathbf{x}}
\newcommand{\rmd}{\mathrm{d}}
\newcommand{\vep}{\varepsilon}
\newcommand{\vphi}{\varphi}
\newcommand{\psihn}[1]{\psi^{#1}}

\usepackage[capitalize,nameinlink,noabbrev]{cleveref}
\crefname{equation}{}{}

\theoremstyle{thmstyletwo}%
\newtheorem{theorem}{Theorem}%  meant for continuous numbers
\newtheorem{proposition}[theorem]{Proposition}%
\newtheorem{remark}{Remark}%

\numberwithin{equation}{section}

\begin{document}

\DOI{DOI HERE}
\copyrightyear{2021}
\vol{00}
\pubyear{2021}
\access{Advance Access Publication Date: Day Month Year}
\appnotes{Paper}
\copyrightstatement{Published by Oxford University Press on behalf of the Institute of Mathematics and its Applications. All rights reserved.}
\firstpage{1}

%\subtitle{Subject Section}

\title[An EWI-FS for the LogSE]{Optimal error bounds on an exponential wave integrator Fourier spectral method for the logarithmic Schr\"odinger equation}

\author{Weizhu Bao
\address{\orgdiv{Department of Mathematics}, \orgname{National University of Singapore}, \orgaddress{\postcode{119076}, \country{Singapore}}}}
\author{Ying Ma
\address{\orgdiv{Department of Mathematics, School of Mathematics, Statistics and Mechanics}, \orgname{Beijing University of Technology, Beijing}, \orgaddress{\postcode{100124}, \state{Beijing}, \country{China}}}}
\author{Chushan Wang*
\address{\orgdiv{Department of Mathematics}, \orgname{National University of Singapore}, \orgaddress{\postcode{119076}, \country{Singapore}}}}

\authormark{Author Name et al.}

\corresp[*]{Corresponding author: \href{chushanwang@u.nus.edu}{chushanwang@u.nus.edu}}

\received{Date}{0}{Year}
\revised{Date}{0}{Year}
\accepted{Date}{0}{Year}

%\editor{Associate Editor: Name}

\abstract{We prove a nearly optimal error bound on the exponential wave integrator Fourier spectral (EWI-FS) method for the logarithmic Schr\"odinger equation (LogSE) under the assumption of $H^2$-solution, which is theoretically guaranteed. Subject to a CFL-type time step size restriction $\tau |\ln \tau| \lesssim h^2/|\ln h|$ for obtaining the stability of the numerical scheme affected by the singularity of the logarithmic nonlinearity, an $L^2$-norm error bound of order $O(\tau |\ln \tau|^2 + h^2 |\ln h|)$ is established, where $\tau$ is the time step size and $h$ is the mesh size.  Compared to the error estimates of the LogSE in the literature, our error bound either greatly improves the convergence rate under the same regularity assumptions or significantly weakens the regularity requirement to obtain the same convergence rate. Moreover, our result can be directly applied to the LogSE with low regularity $L^\infty$-potential, which is not allowed in the existing error estimates. Two main ingredients are adopted in the proof: (i) an $H^2$-conditional $L^2$-stability estimate, which is established using the energy method to avoid singularity of the logarithmic nonlinearity, and (ii) mathematical induction with inverse inequalities to control the $H^2$-norm of the numerical solution. Numerical results are reported to confirm our error estimates and demonstrate the necessity of the time step size restriction imposed. We also apply the EWI-FS method to investigate soliton collisions in one dimension and vortex dipole dynamics in two dimensions. }

\keywords{logarithmic Schr\"odinger equation; exponential wave integrator; low regularity potential; Fourier spectral method; error estimate; voterx dipole.}

% \boxedtext{
% \begin{itemize}
% \item Key boxed text here.
% \item Key boxed text here.
% \item Key boxed text here.
% \end{itemize}}

\maketitle

\section{Introduction}
The logarithmic Schr\"odinger equation (LogSE) arises in a model of nonlinear wave mechanics \citep{LogSE1}, and has found various applications in quantum mechanics, quantum optics, transport and diffusion phenomena, open quantum systems, nuclear physics, and Bose-Einstein condensation (see, e.g., \citep{LogSE2,LogSE3,LogSE4}). In these applications, it is often of particular interest to consider wave propagation in random or disorder medium, which introduces additional low regularity potential into the equation \citep{poten1,poten2,poten3}. In this paper, we consider the following LogSE on a bounded domain $\Omega = \Pi_{j=1}^d (a_j, b_j) \subset \R^d \ (d = 1, 2, 3)$ equipped with periodic boundary condition as
\begin{equation}\label{LogSE}
	\left\{
	\begin{aligned}
		&i \partial_t \psi = -\Delta \psi + V(\vx) \psi + \lambda \ln(|\psi|^2) \psi, && \vx \in \Omega, \quad t>0, \\
		&\psi(\vx, 0) = \psi_0(\vx), && \vx \in \overline{\Omega}, 
	\end{aligned}
	\right.
\end{equation}
where $t \geq 0$ is time, $\vx \in \Omega$ is the spatial coordinate with $\vx = x$ when $d=1$, and $ \psi:=\psi(\vx, t) \in \C $ is the wave function or order parameter. Here, $V \in L^\infty(\Omega)$ is a real-valued (low regularity) potential and $\lambda \in \R$ is a given constant characterizing the nonlinear interaction. The LogSE \cref{LogSE}  conserves the mass
\begin{equation}\label{eq:M}
	M(\psi(\cdot, t)) := \int_\Omega |\psi(\vx, t)|^2 \rmd \vx \equiv M(\psi_0), \quad t \geq 0, 
\end{equation}
and the energy 
\begin{align}\label{eq:E}
	E(\psi(\cdot, t)) 
	&:= \int_\Omega \left[|\nabla \psi(\vx, t)|^2 + V(\vx)|\psi(\vx, t)|^2 + F(|\psi(\vx, t)|^2) \right] \rmd \vx \notag \\
	&\equiv E(\psi_0), \quad t\geq0,  
\end{align}
where $F(\rho) := \lambda \int_0^\rho \ln(r) \rmd r = \lambda (\rho \ln(\rho) - \rho) $ for $\rho \geq 0$. 
%{\color{blue} Noting \cref{eq:M}, the following quantity (which may also be regarded as the energy) is also conserved:}
%\begin{equation*}
%	I(\psi(\cdot, t)) := \int_\Omega \left[|\nabla \psi(\vx, t)|^2 + V(\vx)|\psi(\vx, t)|^2 + \lambda|\psi(\vx, t)|^2 \ln(|\psi(\vx, t)|^2)\right] \rmd \vx. 
%\end{equation*}

A closely related model to the LogSE is the nonlinear Schrödinger equation (NLSE) with power-type nonlinearity: 
\begin{equation}\label{NLSE}
	i \partial_t \psi = -\Delta \psi + V(\vx) \psi + \lambda |\psi|^{2\sigma} \psi, \quad \vx \in \Omega, \quad t>0, 
\end{equation}
where $\lambda \in \R$ and $\sigma \in \R^{+}$, which has been extensively studied \citep{cazenave2003,NLS}. Although the LogSE can be viewed as the limit of the NLSE \cref{NLSE} as $\sigma \rightarrow 0$ (see \cite{LogRemi2022} and \cite{log2019arma} for a detailed discussion of such convergence), a distinctive feature of the LogSE compared to the NLSE \cref{NLSE} is that the nonlinearity $z \rightarrow \ln(|z|^2)z$ is not locally Lipschitz continuous due to the singularity of the logarithm at the origin. Such singularity results in challenges in the analytical study of the LogSE, making even the Cauchy problem fundamentally different from that of the NLSE. The (global) well-posedness of the Cauchy problem of the LogSE \cref{LogSE} has been extensively studied since the first work \citep{log1980}, where weak $H^1$- and $H^2$-solutions are constructed for $\lambda<0$ using compactness arguments. Subsequent works \citep{LogRemi2022,sinum2019} extend these results to cover both $\lambda<0$ and $\lambda > 0$. More recently, strong solutions have been constructed without using compactness arguments \citep{Logstrong2024,Logstrong2018,logcarles2024}. In particular, it remains open whether higher-than-$H^2$-regularity (e.g., $H^3$) can be propagated by the LogSE \cref{LogSE} even without potential due to the singularity of the logarithmic nonlinearity. 

The logarithmic nonlinearity also gives rise to several unique dynamical properties of the LogSE. First, it is shown that when $\Omega = \R^d$ and $V(\vx) \equiv 0$ in \cref{LogSE}, if the initial data is Gaussian, the solution will remain Gaussian for all time, and the dynamics reduces to an ODE system for the parameters of the Gaussian \citep{sinum2019}. Under the same setting for $\Omega$ and $V$, though the nonlinearity $\lambda \ln|\psi|^2$ has no definite sign no matter $\lambda>0$ or $\lambda<0$, it is proved that no solutions are dispersive when $\lambda < 0$, whereas all solutions disperse at a faster rate than the solution of the NLSE \cref{NLSE} when $\lambda > 0$ \citep{Log_Remi}. Another unusual property of the LogSE is that the dynamics are invariant under change of the size of initial data as the equation remains valid under $\psi_0 \rightarrow \kappa \psi_0 $ and $\psi \rightarrow \kappa \psi e^{-i t \lambda \ln|\kappa|^2}$ for $\kappa \in \C$. Additionally, the LogSE satisfies a tensorization property \citep{LogRemi2022}, which is the motivation for introducing this model \citep{LogSE1}.

Along the numerical side, many accurate and efficient numerical methods have been proposed and analyzed for the NLSE \cref{NLSE} with smooth potential and cubic nonlinearity (i.e., $\sigma=1$). These include the finite difference time domain (FDTD) method \citep{FD,bao2013,Ant,henning2017}, the time-splitting method \citep{BBD,lubich2008,schratz2016,Ant,splitting_low_reg,su2022,bao2023_semi_smooth,splittinglowregfull}, the exponential wave integrator (EWI) \citep{ExpInt,SymEWI,bao2023_EWI}, and the low regularity integrator (LRI) \citep{LRI,tree1,tree2,tree3} designed for the NLSE with extremely rough initial data. Most of these methods can be applied to solving the LogSE with/without proper regularization of the logarithmic nonlinearity, such as FDTD methods \citep{sinum2019,logFD22,logFD23,log2024} and time-splitting methods \citep{bao2019,bao2022,zhang2024}. However, the error estimates of these methods for the LogSE \cref{LogSE} is a subtle and challenging question due to the singularity of the nonlinearity. For the FDTD method, first-order convergence in $L^2$-norm is obtained under the assumption, among others, $\partial_{tt} \psi \in L^2(\Omega)$ \citep{log2024,logFD23}, which generally requires $H^4$-solution by the equation. However, such an assumption is already beyond the well-posedness theory of the LogSE, and cannot be satisfied in general, especially when there is low regularity potential. For the time-splitting methods, half-order convergence (up to a logarithmic factor) is established when $V \in H^1(\Omega) \cap L^\infty(\Omega)$ under the assumption of $H^2$-solution of the LogSE which is theoretically guaranteed \citep{bao2019}. However, this result does not allow purely $L^\infty$-potential, and the convergence order reduction from first-order to half-order is not observed in the numerical experiments. Hence, it remains unclear whether first-order temporal convergence can be achieved for any time discretizations under the assumption of $L^\infty$-potential and $H^2$-solution of the LogSE. 

Very recently, for the NLSE \cref{NLSE} with $L^\infty$-potential and $C^1$-nonlinearity (satisfied for any $\sigma >0$), optimal first-order $L^2$-norm error bounds are established under the assumption of $H^2$-solutions by the same authors for both time-splitting methods \citep{bao2023_improved} and EWIs \citep{bao2023_EWI,bao2023_sEWI}. Considering that the LogSE \cref{LogSE} can be understood as the limit of \cref{NLSE} as $\sigma \rightarrow 0$, it is natural to expect the same optimal error bounds (up to some logarithmic factor) to hold for the LogSE. However, {due to the singularity of the nonlinearity again}, the error estimates in \citep{bao2023_improved,bao2023_EWI,bao2023_sEWI} cannot be directly applied and new analysis techniques are needed. In fact, as we shall show in the current work, this limit cannot be trivially taken and some CFL-type time step size restriction is needed to compensate for the singularity of the nonlinearity. 

In this work, we introduce an exponential wave integrator Fourier spectral (EWI-FS) method to solve the LogSE \cref{LogSE}. The use of the EWI-FS method is motivated by existing works on the NLSE with low regularity potential and nonlinearity \citep{bao2023_semi_smooth,bao2023_EWI,bao2023_improved,bao2023_sEWI,bao2024,lin2024}, where it is shown that (i) the EWI is advantageous over time-splitting methods under low regularity potential and nonlinearity, and (ii) the Fourier spectral method is able to achieve optimal spatial convergence consistent with the regularity of the exact solution. In fact, the optimal spatial convergence is also crucial in obtaining the temporal convergence order. As a result, for the EWI-FS method, we prove a nearly optimal error bound of $O(\tau |\ln \tau|^{2} + h^2 |\ln h|)$ with $\tau$ being the time step size and $h$ being the mesh size, under the assumption of $L^\infty$-potential and $H^2$-solution of the LogSE \cref{LogSE}, and subject to a CFL-type time step size restriction $\tau |\ln \tau| \lesssim h^2/|\ln h|$ (see \cref{thm:main}). This time step size restriction is necessary in the practical implementation of the EWI-FS method as justified by the numerical results, and this is purely due to the singularity of the nonlinearity instead of the low regularity potential. To our best knowledge, it is the first work that establishes (nearly) first-order temporal convergence {and second-order spatial convergence} for the LogSE under the assumption of $H^2$-solution. Compared to the results for FDTD methods, our error bound significantly relaxes the regularity requirement on both the potential and exact solution for first-order temporal convergence. Compared to the results for time-splitting methods, our error bound improves the convergence order and weakens the regularity requirement on potential simultaneously. 

The remainder of this paper is structured as follows. In \cref{sec:2}, we introduce the first-order EWI and its spatial discretization by the Fourier spectral method, and state our main error estimate result. The proof of the main result is presented in \cref{sec:3}. Extensive numerical results are provided to validate our error estimates and to study the dynamics of the LogSE in \cref{sec:num}. Finally, some concluding remarks and directions for future research are provided in \cref{sec:5}. Throughout the paper, standard notations of Sobolev spaces and corresponding norms are adopted. We denote by $ C $ a generic positive constant independent of the time step size $ \tau $ and the mesh size $ h $, and by $ C(\alpha) $ a generic positive constant depending on the parameter $ \alpha $. The notation $ A \lesssim B $ is used to represent that there exists a generic constant $ C>0 $, such that $ |A| \leq CB $.

\section{Exponential wave integrator Fourier spectral method and main results}
\label{sec:2}
In this section, we introduce the exponential wave integrator Fourier spectral method (EWI-FS) to solve the LogSE \cref{LogSE}, and present our main results. Here, we directly approximate the LogSE \cref{LogSE} without regularizing the logarithmic nonlinearity as in \cite{logFD23}, \cite{log2024} and \cite{zhang2024}. For simplicity of the presentation, we only present the numerical scheme in one dimension (1D) with $\Omega = (a, b)$. Generalizations to two dimensions (2D) and three dimensions (3D) are straightforward. We shall frequently use the periodic Sobolev spaces defined as
\begin{equation}
	H^m_\text{per}(\Omega) = \{\phi \in H^m(\Omega): \phi^{(k)}(a) = \phi^{(k)}(b), \  k = 0, \cdots, m-1 \}, \quad m \in \mathbb{Z}^+. 
\end{equation}

\subsection{EWI-FS method}	
In the following, we present the EWI-FS method. We first discretize the LogSE \cref{LogSE} in space by the Fourier spectral method to obtain a coupled system of ODEs. Then we use a first-order EWI to further discretize the ODE system in time. To simplify the notation, we denote $ \psi(\cdot, t)$ by $\psi(t)$ and define an operator $B$ as
\begin{equation}
	B(\phi) = V \phi + \lambda \ln(|\phi|^2)\phi, \quad \phi \in L^2(\Omega). 
\end{equation}
Choose a mesh size $h = (b-a)/N$ with $N$ being a positive even integer and denote the grid points as
\begin{equation*}
	x_j = a+jh, \quad j \in \mathcal{T}_N^0 := \{0, \cdots, N\}. 
\end{equation*}
Define the index set of frequency as
\begin{equation}
	\mathcal{T}_N = \left\{ - \frac{N}{2}, \cdots, \frac{N}{2} - 1 \right\}, 
\end{equation}
and denote
\begin{equation}
	X_N = \text{span}\{e^{i \mu_l (x-a)}: l \in \mathcal{T}_N\}, \quad \mu_l = \frac{2 \pi l}{b-a}. 
	%		& Y_N = \{(v_0, \cdots, v_N)^T \in \mathbb{C}^{N+1}: v_0 = v_N\}. 
\end{equation}
Let $P_N:L^2(\Omega) \rightarrow X_N$ be the $L^2$ projection on $X_N$ defined for any $\phi \in L^2(\Omega)$ as
\begin{equation}
	(P_N \phi)(x) = \sum_{l \in \mathcal{T}_N} \widehat{\phi}_l e^{i \mu_l(x-a)}, \quad x \in \Omega, 
\end{equation}
where $\widehat{\phi}$ is the Fourier transform of $\phi$ defined by
\begin{equation}
	\widehat{\phi}_l = \frac{1}{b-a} \int_a^b \phi(x) e^{-i \mu_l(x-a)} \rmd x, \quad l \in \Z. 
\end{equation}
The Fourier spectral discretization of the LogSE \cref{LogSE} reads: Find
\begin{equation}
	\psi_N = \psi_N(t) = \psi_N(x, t) = \sum_{l \in \mathcal{T}_N} \widehat{(\psi_N)}_l(t) e^{i \mu_l(x-a)} \in X_N, \quad t \geq 0, 
\end{equation}
such that $\psi_N(0) = P_N \psi_0$ and 
\begin{equation}\label{eq:PNLogSE-eq}
	i \partial_t \psi_N(t) = -\Delta \psi_N(t) + P_N B(\psi_N(t)), \quad  t>0. 
\end{equation}
By the orthogonality, we then obtain the equations of the Fourier coefficients as
\begin{equation}\label{eq:PNLogSE}
	\left\{
	\begin{aligned}
		&i \frac{\rmd}{\rmd t} \widehat{(\psi_N)}_l(t) = \mu_l^2 \widehat{(\psi_N)}_l(t) + \widehat{(B(\psi_N))}_l, \quad l \in \mathcal{T}_N, \quad  t>0, \\
		&\widehat{(\psi_N)}_l(0) = \widehat{(\psi_0)}_l, \quad l \in \mathcal{T}_N.  
	\end{aligned}
	\right.
\end{equation}
Here, with the understanding that $z \ln|z|^2 = 0$ when $z=0$, we note that for $\varsigma > 0$ arbitrarily small, 
\begin{equation}
	|z \ln|z|^2| \lesssim |z|^{1+\varsigma} + |z|^{1-\varsigma}, \quad z \in \C, 
\end{equation}
which implies $B(\phi) \in L^2(\Omega)$ for any $\phi \in X_N$ and thus its $L^2$-projection and Fourier transform in \cref{eq:PNLogSE-eq,eq:PNLogSE} are well-defined. 
Then we further discretize \cref{eq:PNLogSE} in time by a first-order EWI. Choose a time step size $\tau>0$ and denote time steps as $t_n = n \tau$ for $n=0, 1, \cdots$. By the Duhamel's formula, the exact solution of \cref{eq:PNLogSE} satisfies
\begin{equation}\label{eq:Duhamel}
	\widehat{(\psi_N)}_l(t_{n+1}) = e^{-i\tau\mu_l^2} \widehat{(\psi_N)}_l(t_n) - i \int_0^\tau e^{-i(\tau - s)\mu_l^2} B_l^n(s) \rmd s, \quad l \in \mathcal{T}_N, 
\end{equation}
where $B_l^n(s) = \widehat{(W(s))}_l$ with $W(s) = B(\psi_N(t_n + s))$ for $0 \leq s \leq \tau$. 
Adopting the approximation $B_l^n(s) \approx B_l^n(0)$ in the integral above and integrating out $e^{-i(\tau-s)\mu_l^2}$ exactly, we obtain
\begin{align}
	\widehat{(\psi_N)}_l(t_{n+1})
	&\approx e^{-i\tau\mu_l^2} \widehat{(\psi_N)}_l(t_n) - i \int_0^\tau e^{-i(\tau - s)\mu_l^2} \rmd s B^n_l(0) \notag \\
	&= e^{-i\tau\mu_l^2} \widehat{(\psi_N)}_l(t_n) - i \tau \vphi_1(-i \tau \mu_l^2) B^n_l(0), \quad l \in \mathcal{T}_N,  
\end{align}
where $\vphi_1(z) = (e^z-1)/z$ for $z \in \C$. This naturally leads to the following numerical scheme: Let $\psihn{n}(\cdot) \in X_N$ be the numerical approximation to $\psi(\cdot, t_n)$ for $n \geq 0$, then the EWI-FS method reads
\begin{equation}\label{eq:EWI-FS_scheme}
	\begin{aligned}
		&\widehat{(\psihn{n+1})}_l = e^{-i\tau\mu_l^2} \widehat{(\psihn{n})}_l - i \tau \vphi_1(- i \tau \mu_l^2) \widehat{B(\psihn{n})}_l, \quad l \in \mathcal{T}_N, \quad n \geq 0, \\
		&\widehat{(\psihn{0})}_l = \widehat{(\psi_0)}_l, \quad l \in \mathcal{T}_N. 
	\end{aligned}
\end{equation}
Rewriting \cref{eq:EWI-FS_scheme} in the physical space, we obtain that $\psihn{n} \in X_N$ satisfies
\begin{equation}\label{eq:EWI-FS}
	\begin{aligned}
		&\psihn{n+1} = e^{i\tau\Delta} \psihn{n} - i \tau \vphi_1(i \tau \Delta) P_N B(\psihn{n}), \quad n \geq 0, \\
		&\psihn{0} = P_N \psi_0, 
	\end{aligned}
\end{equation}
where $\vphi_1(i\tau\Delta)$ is defined through its action in the Fourier space (see \cite{bao2023_EWI}). 

\begin{remark}\label{rem:bc}
	When $\Omega$ is equipped with the homogeneous Dirichlet boundary condition or Neumann boundary condition, an exponential wave integrator sine spectral method or cosine spectral method similar to \cref{eq:EWI-FS_scheme} is straightforward \citep{Ant}, and the main result can be directly generalized to both cases. 
\end{remark}

\subsection{Main results}
In this subsection, we state our main error estimate results for the EWI-FS method \cref{eq:EWI-FS_scheme} applied to the LogSE \cref{LogSE}. 
%	Let $0<T<T_\text{max}$ with $T_\text{max}$ being the maximum existing time of the exact solution of the LogSE \cref{LogSE}. 
According to the known $H^2$ well-posedness of the LogSE \cref{LogSE} \citep{Log_Remi,sinum2019,kato1987,Logstrong2024}, we make the following assumptions on the exact solution: For some $T>0$, 
\begin{equation}\label{eq:regularity}
	\psi \in C([0, T]; H^2_\text{per}(\Omega)) \cap C^1([0, T]; L^2(\Omega)). 
\end{equation}
In fact, the $H^2$ well-posedness is proved for the LogSE without potential, i.e., \cref{LogSE} with $V(\vx) \equiv 0$ (see, e.g., \citet[Theorem 1.2]{Logstrong2024} and \citet[Theorem 2.2]{sinum2019}) and for the NLSE with $L^\infty$-potential and power-type nonlinearity (see, e.g., \citet[Theorem 2]{kato1987}). Moreover, it is the highest regularity that can be theoretically guaranteed for the LogSE: In the absence of the potential (i.e., $V(\vx) \equiv 0$), it remains open if higher regularity (e.g., $H^3$) can be propagated even locally in time \citep{LogRemi2022,logremi2024}. Hence, it is crucial to establish error estimates under the assumption \cref{eq:regularity} of $H^2$-solution of the LogSE. 

We define a constant
\begin{equation}\label{eq:M_def}
	M: = \max\{\| \psi \|_{L^\infty([0, T]; H^2)}, \| \partial_t \psi \|_{L^\infty([0, T]; L^2)}, \| \psi \|_{L^\infty([0, T]; L^\infty)}, \| V \|_{L^\infty}\}. 
\end{equation}

\begin{theorem}\label{thm:main}
	Under the assumptions $V \in L^\infty(\Omega)$ and $\psi \in C([0, T]; H^2_\text{\rm per}(\Omega)) \cap C^1([0, T]; L^2(\Omega))$, there exists $0<\tau_0, \  h_0<e^{-1}$ sufficiently small depending on $M$ and $T$ such that when $0<\tau<\tau_0$, $0<h<h_0$ and {$\tau|\ln \tau| \leq \tilde C h^{2}/|\ln h|$ for some $\tilde C>0$}, we have
	\begin{align*}
		&\| \psi(t_n) - \psihn{n} \|_{L^2} \lesssim \tau |\ln \tau |^{2} + h^2 |\ln h |, \quad  \| \psihn{n} \|_{H^2} \lesssim |\ln h|, \\
		&\| \psi(t_n) - \psihn{n} \|_{H^1} \lesssim \sqrt{\tau} |\ln \tau | + h |\ln h |,
		\quad 0 \leq n \leq T/\tau. 
	\end{align*}
\end{theorem}

\begin{remark}\label{rem:highregularity}
	If $\psi \in C([0, T]; H^m_\text{\rm per}(\Omega))$ for some $m>2$ in \cref{thm:main}, {under a slightly relaxed time step size restriction $\tau|\ln \tau| \leq \tilde C h^2$ with $\tilde C$ sufficiently small}, the error bounds can be improved to
	\begin{align*}
		&\| \psi(t_n) - \psihn{n} \|_{L^2} \lesssim \tau |\ln \tau | + h^m |\ln h|, \quad  \| \psihn{n} \|_{H^2} \leq 1+M, \\
		&\| \psi(t_n) - \psihn{n} \|_{H^1} \lesssim \sqrt{\tau |\ln \tau |} + h^{m-1} |\ln h|,
		\quad 0 \leq n \leq T/\tau. 
	\end{align*}
	Since the above results can be obtained in a manner analogous to but simpler than \cref{thm:main}, we shall omit the proof.  
\end{remark}

%	\begin{remark}\label{rem:highregularity}
	%		If, in addition, $\psi \in C([0, T]; H^m(\Omega))$ for some real $m>2$ in \cref{thm:main}, then the time step size restriction can be relaxed to $\tau|\ln \tau| \leq h^2/|\ln h|$, and the error estimates can be improved to
	%		\begin{align*}
		%			&\| \psi(t_n) - \psihn{n} \|_{L^2} \lesssim \tau |\ln \tau | + h^2, \quad  \| \psihn{n} \|_{H^2} \leq 1+M, \\
		%			&\| \psi(t_n) - \psihn{n} \|_{H^1} \lesssim \sqrt{\tau |\ln \tau |} + h,
		%			\quad 0 \leq n \leq \frac{T}{\tau}. 
		%		\end{align*}
	%		Since the above results can be obtained in a manner analogous to but simpler than \cref{thm:main}, we shall omit the proof.  
	%	\end{remark}

The time step size restrictions in \cref{thm:main,rem:highregularity} are essentially the CFL condition when ignoring the logarithmic factors. This time step size restriction is natural in terms of the balance between temporal and spatial errors, and, notably, it can be observed in our numerical experiments, indicating its necessity (see \cref{sec:num}). This should be compared to the results in \citep{bao2023_EWI,bao2023_sEWI} for the EWI-FS method applied to the NLSE with $L^\infty$-potential and power-type nonlinearity, where optimal error bounds can be obtained without any CFL-type time step size restriction. Hence, although the logarithmic nonlinearity can be viewed as the limiting case of the power-type nonlinearity, 
%	\begin{equation}
	%		\lim_{\sigma \rightarrow 0^+} \frac{\rho^\sigma - 1}{\sigma} = \ln (\rho), \quad \rho > 0, 
	%	\end{equation}
the numerical method may behave significantly differently in these two cases. 
%{\color{blue} In addition, the time step size restriction in \cref{thm:main} cannot be relaxed to $ \tau|\ln \tau| \leq C h^{2}/|\ln h| $ with $C > 1$ as can be seen from \cref{eq:induction_H2} in the proof.}

%	\subsection{The sEWI}
%	The symmetric exponential wave integrator Fourier spectral method (sEWI-FS) method reads
%	\begin{equation}\label{eq:sEWI-FS}
	%		\begin{aligned}
		%			&\psihn{n+1} = e^{2i\tau\Delta} \psihn{n-1} - 2i \tau e^{i \tau \Delta} \vphi_s(\tau \Delta) P_N \left( f(|\psihn{n}|^2)\psihn{n} \right), \quad n \geq 1, \\
		%			&\psihn{1} = P_N \psi(\cdot, t_1), \quad \psihn{0} = P_N \psi_0. 
		%		\end{aligned}
	%	\end{equation}
%	For simplicity of the error estimates, we assume that $\psihn{1}$ is exactly given. In practice, one can use \cref{eq:EWI-FS} to compute the first step. 

\section{Error estimates}\label{sec:3}
In this section, we prove the main result \cref{thm:main}. 

\subsection{Estimates for the nonlinearity}
We first introduce some estimates for the logarithmic nonlinearity. For $0 < \vep < 1$, we define $g:\C \rightarrow \C$ and $g_\vep:\C \rightarrow \C$ as
\begin{equation}
	g(z) = z\ln(|z|^2), \quad  g_\vep(z) = z\ln(|z| + \vep)^2, \qquad z \in \C. 
\end{equation}
%	\begin{equation}
	%		f_\vep(\rho) = \lambda \ln(\sqrt{\rho}+\vep)^2, \qquad \rho \geq 0. 
	%	\end{equation}
In fact, $g_\vep$ can be regarded as a regularization of $g$, which has better regularity, and is, in particular, Lipschitz continuous \citep{sinum2019,zhang2024}. 
For any $\phi \in L^2(\Omega)$, we define, $g(\phi) (x) := g(\phi(x))$ and $g_\vep(\phi) (x) := g_\vep(\phi(x))$ for $x \in \Omega$. By some elementary calculation, we have, for $0 < \vep < 1$, 
\begin{equation}\label{eq:ggep}
	|g(z) - g_\vep(z)| \leq 2\vep, \quad z \in \C. 
\end{equation}
Using $g_\vep$ as an intermediary and leveraging \cref{eq:ggep}, we have, for any $0< \vep < 1$,  
\begin{equation}\label{lem:diff_g}
	|g(z_1) - g(z_2)| \leq 4\vep + 2(1+L_\vep(M_0))|z_1 - z_2|, \quad z_1, z_2 \in \C,  
\end{equation}
where $M_0 = \max\{|z_1|, |z_2|\}$ and $L_\vep(s) = \max\{|\ln(\vep)|, \ln(1+s)\}$ for $s \geq 0$. The proof of \cref{eq:ggep,lem:diff_g} can be found in \citet[Lemma 3.1]{zhang2024} and is omitted here. 

Then we recall the following algebraic property of the nonlinearity $g$ first discovered in \cite{log1980}: 
\begin{equation}\label{lem:lip}
	|\text{\rm Im}[(g(z_1) - g(z_2))(\overline{z_1 - z_2})]| \leq 2 |z_1 - z_2|^2, \quad z_1, z_2 \in \C. 
\end{equation}
The estimate \cref{lem:lip} is crucial in overcoming the singularity of the logarithmic nonlinearity, and have been successfully used in establishing the well-posedness of the LogSE \citep{log1980,sinum2019,Logstrong2024} and in the error estimates of time-splitting methods \citep{bao2019,bao2022,zhang2024} and FDTD methods \citep{logFD23,log2024} for the LogSE. 

\subsection{Local truncation error}
We define the numerical flow $\Phi_h^t: X_N \rightarrow X_N$ associated with the EWI-FS method \cref{eq:EWI-FS} as
\begin{equation}\label{eq:Phi_def}
	\Phi_h^t(\phi) = e^{i t \Delta} \phi - i t \vphi_1(it\Delta) P_N B(\phi), \quad \phi \in X_N, \quad t \geq 0.   
\end{equation}
Define the local truncation error 
\begin{equation}\label{eq:mcalE_def}
	\mathcal{E}^n = P_N \psi(t_{n+1}) - \Phi_h^\tau(P_N \psi(t_n)), \quad 0 \leq n \leq T/\tau-1. 
\end{equation}
Then we have the following estimate of the local truncation error. 
\begin{proposition}[Local truncation error]\label{prop:LTE}
	For $0 < \tau < 1 $ and $ 0 < h < 1$, we have
	\begin{equation*}
		\| \mathcal{E}^n \|_{L^2} \lesssim (1+ |\ln(\tau)|) \tau^2 + (1+|\ln(h)|) \tau h^2, \quad  0 \leq n \leq T/\tau -1, 
	\end{equation*}
	where the constant depends on $M$. 
\end{proposition}

\begin{proof}
	By Duhamel's principle, we have
	\begin{equation}\label{eq:duhamel}
		\psi(t_{n+1}) = e^{i \tau \Delta} \psi(t_n) - i \int_0^\tau e^{i(\tau - s)\Delta} B(\psi(t_n + s)) \rmd s.
	\end{equation}
	Applying $P_N$ on both sides of \cref{eq:duhamel}, we have
	\begin{equation}\label{eq:psi}
		P_N \psi(t_{n+1}) = e^{i \tau \Delta} P_N \psi(t_n) - i \int_0^\tau e^{i(\tau - s)\Delta} P_N B(\psi(t_n + s)) \rmd s.  
	\end{equation}
	Recall the construction of the EWI-FS method, we have
	\begin{equation}\label{eq:psin}
		\Phi_h^\tau(P_N \psi(t_n)) = e^{i\tau\Delta} P_N \psi(t_n) - i \int_0^\tau e^{i(\tau-s)\Delta} P_N B(P_N \psi(t_n)) \rmd s. 
	\end{equation}
	Subtracting \cref{eq:psin} from \cref{eq:psi}, and recalling \cref{eq:mcalE_def}, we obtain
	\begin{align}\label{eq:LTE_decomp}
		\mathcal{E}^n 
		&= -i\int_0^\tau e^{i(\tau - s)\Delta} P_N \left( B(\psi(t_n + s)) - B(P_N\psi(t_n)) \right) \rmd s \notag \\
		&= -i\int_0^\tau e^{i(\tau - s)\Delta} P_N \left( B(\psi(t_n + s)) - B(\psi(t_n)) \right) \rmd s \notag \\
		&\quad -i\int_0^\tau e^{i(\tau - s)\Delta} P_N \left( B(\psi(t_n)) - B(P_N\psi(t_n)) \right) \rmd s =: r^n + r^n_h. 
	\end{align}
	By \cref{lem:diff_g} and the standard projection error estimate of $P_N$, we have, for any $0<\vep<1$, 
	\begin{align}
		\| r^n_h \|_{L^2} 
		&\leq \tau \| B(\psi(t_n)) - B(P_N\psi(t_n)) \|_{L^2} \notag \\
		&\leq \tau \left(\| V \|_{L^\infty} \| \psi(t_n) - P_N \psi(t_n) \|_{L^2} + \| g(\psi(t_n)) - g(P_N\psi(t_n)) \|_{L^2}\right) \notag \\
		&\lesssim \tau h^2 + \tau \vep + \tau (1+L_\vep(M)) \| \psi(t_n) - P_N\psi(t_n) \|_{L^2} \notag \\
		&\lesssim \tau \vep + (1+|\ln(\vep)| + \ln(1+M)) \tau h^2, 
	\end{align}
	which implies, by choosing $ \vep = h^2 $, 
	\begin{equation}\label{eq:rh_est}
		\| r^n_h \|_{L^2} \lesssim \tau (1 + |\ln(h)|)h^2. 
	\end{equation}
	Similarly, by \cref{lem:diff_g}, we have
	\begin{align}
		\| r^n \|_{L^2} 
		&\leq \int_0^\tau \| B(\psi(t_n + s)) - B(\psi(t_n)) \|_{L^2} \rmd s \notag \\
		&\lesssim \tau \vep + (1+\| V \|_{L^\infty}+L_\vep(M)) \int_0^\tau \| \psi(t_n+s) - \psi(t_n) \|_{L^2} \rmd s \notag \\
		&\lesssim \tau \vep + (1+|\ln(\vep)| + \ln(1+M))\tau^2 \| \partial_t \psi \|_{L^\infty([t_n, t_{n+1}]; L^2(\Omega))}, 
	\end{align}
	which implies, by choosing $\vep = \tau$, 
	\begin{equation}\label{eq:r_est}
		\| r^n \|_{L^2} \lesssim (1+|\ln(\tau)|)\tau^2, 
	\end{equation}
	which together with \cref{eq:rh_est} completes the proof by \cref{eq:LTE_decomp}. 
\end{proof}

\subsection{$H^2$-conditional $L^2$-stability}
In this subsection, we use the energy method together with \cref{lem:lip} to establish an $H^2$-conditional $L^2$-stability of the EWI-FS method \cref{eq:EWI-FS}. Due to the singularity (and the lack of Lipschitz continuity) of the logarithmic nonlinearity, it is essential to use energy estimates to get rid of the exponential dependence on the unbounded Lipschitz constant of the nonlinearity. 

Recalling \cref{eq:Phi_def}, we have the following estimate of the numerical flow $\Phi^\tau_h$. 
\begin{proposition}[Stability]\label{prop:stability}
	Let $v_0, w_0 \in X_N$ such that $\| v_0 \|_{H^\frac{7}{4}} \leq M_0$, $\| v_0 \|_{H^2} \leq M_1$ and $\| w_0 \|_{H^2} \leq M_2$. We have, for $0<\tau<e^{-1}$, 
	\begin{equation*}
		\| \Phi_h^\tau(v_0) - \Phi_h^\tau(w_0) \|_{L^2} \leq e^{C_\text{\rm s}\tau} \| v_0-w_0 \|_{L^2} + C(M_0) \tau^2 |\ln \tau| (1+M_1) + C(M_2) \tau^2 |\ln \tau|, 
	\end{equation*}
	where $C_\text{\rm s} = 2|\lambda| + \| V \|_{L^\infty}$. 
\end{proposition}

\begin{proof}
	Let $v(t) = \Phi_h^t(v_0)$ and $w(t) = \Phi_h^t(w_0)$ for $0 \leq t \leq \tau$. Then $v(t)$ and $w(t)$ satisfy the following linear inhomogeneous Schr\"odinger equations as
	\begin{equation}\label{eq:SE_v}
		i \partial_t v(t) = -\Delta v(t) + P_N B(v_0), \quad 0 \leq t \leq \tau, \quad v(0) = v_0 \in X_N, 
	\end{equation}
	and
	\begin{equation}\label{eq:SE_w}
		i \partial_t w(t) = -\Delta w(t) + P_N B(w_0), \quad 0 \leq t \leq \tau, \quad w(0) = w_0 \in X_N. 
	\end{equation}
	Define $e(t) = v(t) - w(t) \in X_N$. Subtracting \cref{eq:SE_w} from \cref{eq:SE_v}, one obtains
	\begin{align}\label{stability_e}
		i \partial_t e(t) 
		&= -\Delta e(t) + P_N B(v_0) - P_N B(w_0) \notag \\
		&= -\Delta e(t) + P_N B(v_0) - P_N B(v(t)) + P_N B(v(t))  - P_N B(w(t)) \notag \\
		&\quad + P_N B(w(t))  - P_N B(w_0) \notag \\
		&=  -\Delta e(t) + P_N B(v(t))  - P_N B(w(t)) - R_1 + R_2, 
	\end{align}
	where
	\begin{equation}\label{eq:R_def}
		R_1 = P_N B(v(t)) - P_N B(v_0), \quad R_2 = P_N B(w(t))  - P_N B(w_0). 
	\end{equation}
	Multiplying $\overline{e(t)}$ on both sides of \cref{stability_e}, taking the imaginary part, and integrating over $\Omega$, we get
	\begin{align}\label{stability_e_2}
		\frac{1}{2} \frac{\rmd}{\rmd t} \| e(t) \|_{L^2}^2 
		&= \text{Im} \langle P_N B(v(t))  - P_N B(w(t)), e(t) \rangle \notag \\
		&\quad - \text{Im} \langle R_1(t), e(t) \rangle + \text{Im} \langle R_2(t), e(t) \rangle. 
	\end{align}
	By \cref{lem:lip}, recalling that $e \in X_N$, we have
	\begin{align}\label{r1}
		\text{Im} \langle P_N B(v(t))  - P_N B(w(t)), e(t) \rangle 
		&= \text{Im} \langle B(v(t))  - B(w(t)), e(t) \rangle \notag \\
		&\leq (\| V \|_{L^\infty} + 2|\lambda|) \| e(t) \|_{L^2}^2 = C_\text{s} \| e(t) \|_{L^2}^2, 
	\end{align}
	which implies from \cref{stability_e_2} by Cauchy-Schwartz inequality that
	\begin{equation}\label{e_eq}
		\frac{\rmd}{\rmd t} \| e(t) \|_{L^2} \leq C_\text{s} \| e(t) \|_{L^2} + \| R_1(t) \|_{L^2} + \| R_2(t) \|_{L^2}. 
	\end{equation}
	From \cref{eq:R_def}, using the $L^2$-projection property of $P_N$ and  \cref{lem:diff_g}, we have
	\begin{align}\label{R1_est}
		\| R_1 \|_{L^2} 
		&\leq \| B(v(t))  - B(v_0) \|_{L^2} \notag \\
		&\lesssim \vep +  (\| V \|_{L^\infty}+1+L_\vep(\| v \|_{L^\infty([0, \tau]; L^\infty)}))\sup_{0 \leq t \leq \tau} \| v(t) - v_0 \|_{L^2}.
	\end{align}
	We first estimate $\| v(t) \|_{L^\infty}$. Recalling $v(t)=\Phi_h^t(v_0)$ and \cref{eq:Phi_def}, we get
	\begin{equation}\label{v}
		v(t) = e^{it\Delta}v_0 - it\vphi_1(it\Delta) P_N B(v_0), \quad 0 \leq t \leq \tau. 
	\end{equation}
	For any $\phi \in L^2(\Omega)$ and $t > 0$, we have (see \citet[Lemma 3.9]{bao2023_EWI})
	\begin{equation}\label{lem:vphi}
		\| \vphi_1(i t \Delta) \phi \|_{H^\alpha} \lesssim t^{-\alpha/2} \| \phi \|_{L^2}, \quad 0 \leq \alpha \leq 2. 
	\end{equation}
	From \cref{v}, using \cref{lem:vphi}, the Sobolev embedding $H^\frac{7}{4} \hookrightarrow L^\infty$, the isometry property of $e^{it\Delta}$, we have, when $\tau < \tau_0$, 
	\begin{align}\label{R1_est_1}
		\| v(t) \|_{L^\infty} 
		&\lesssim \| v(t) \|_{H^\frac{7}{4}} \leq \| v_0 \|_{H^\frac{7}{4}} + C t^\frac{1}{8} \| P_N B(v_0) \|_{L^2} \notag \\
		&\leq \| v_0 \|_{H^\frac{7}{4}} + C t^\frac{1}{8} C(\| v_0 \|_{L^\infty}) \leq C(M_0).  
	\end{align}
	Then, using \cref{v} and \cref{lem:vphi} with $\alpha=0$, and the following estimate (see \citep{bao2022})
	\begin{equation}
		\| (e^{i t \Delta} - I) \phi \|_{L^2} \leq t \| \Delta \phi \|_{L^2}, \qquad \phi \in H^2_\text{per}(\Omega), 
	\end{equation}
	we have
	\begin{align}\label{R1_est_2}
		\| v(t) - v_0 \|_{L^2} 
		&\leq \| (e^{it\Delta} - I) v_0 \|_{L^2} + t \| \vphi_1(it\Delta) P_N B(v_0) \|_{L^2} \notag \\
		&\leq \tau \| \Delta v_0 \|_{L^2} + \tau \| B(v_0) \|_{L^2} \notag \\
		&\leq \tau M_1 + \tau \| V \|_{L^\infty} \| v_0 \|_{L^2} + \tau C(\| v_0 \|_{L^\infty}) \notag \\
		&\leq \tau M_1 + \tau C(M_0). 
	\end{align}
	From \cref{R1_est}, using \cref{R1_est_1,R1_est_2} and taking $\vep = \tau$, we obtain
	\begin{align}\label{R1_final}
		\| R_1 \|_{L^2} 
		&\lesssim \tau + (1+L_\tau(C(M_0))) \tau (M_1 + C(M_0)) \notag \\
		&\lesssim C(M_0) \tau(1 + |\ln \tau|) (1+M_1) \leq C(M_0) \tau |\ln \tau| (1+M_1).  
	\end{align}
	The same procedure yields
	\begin{equation}\label{R2_final}
		\| R_2 \|_{L^2} \lesssim C(M_2) \tau |\ln \tau|. 
	\end{equation}
	From \cref{e_eq}, using Gronwall's inequality, noting \cref{R1_final,R2_final}, we get
	\begin{equation}
		\| e(t) \|_{L^2} \leq e^{C_\text{s}\tau} \| e(0) \|_{L^2} + C(M_0) \tau^2 |\ln \tau| (1+M_1) + C(M_2) \tau^2 |\ln \tau|, 
	\end{equation}
	which completes the proof. 
\end{proof}

\subsection{Proof of the main result}
With \cref{prop:LTE,prop:stability}, we are able to obtain the global error estimate of the EWI-FS method \cref{eq:EWI-FS} by mathematical induction and the inverse inequalities \citep{book_spectral}
\begin{equation}\label{eq:inv}
	\| \phi \|_{H^\alpha} \leq C_\text{inv} h^{-\alpha} \| \phi \|_{L^2}, \qquad \phi \in X_N, \quad \alpha > 0. 
\end{equation}

\begin{proof}[Proof of \cref{thm:main}]
	Define the error function $e^n = P_N \psi(t_n) - \psihn{n} \in X_N$ for $0 \leq n \leq T/\tau$. By the standard projection error estimate of $P_N$,  it suffices to prove the error bound for $ e^n $. Note that $e^0 = 0$. When $0 \leq n \leq  T/\tau-1$, we have
	\begin{align}\label{eq:error}
		e^{n+1} 
		&= P_N \psi(t_{n+1}) - \Phi_h^\tau(\psihn{n}) \notag \\
		&= P_N \psi(t_{n+1}) - \Phi_h^\tau(P_N \psi(t_n)) + \Phi_h^\tau(P_N \psi(t_n)) - \Phi_h^\tau(\psihn{n}). 
	\end{align}
	From \cref{eq:error}, using \cref{prop:LTE,prop:stability}, we have, if $\| \psihn{n} \|_{H^\frac{7}{4}} \leq M_0$, 
	\begin{align}\label{eq:error_eq}
		\| e^{n+1} \|_{L^2} 
		&\leq e^{C_\text{s}\tau} \| e^n \|_{L^2} + C(M_0) \tau^2 |\ln \tau| + C(M_0) \tau^2 |\ln \tau|\| \psihn{n} \|_{H^2} \notag \\
		&\quad +  C(M) \tau^2 |\ln(\tau)| + C(M) \tau h^2 |\ln(h)|, \quad 0 \leq n \leq T/\tau-1.
	\end{align}
	%		It follows that, if $\| \psihn{n} \|_{H^\frac{7}{4}} \leq M_0$, 
	%		\begin{align}
		%			\| e^{n+1} \|_{L^2} 
		%			&\leq e^{C\tau} \| e^n \|_{L^2} + \tau^2 (1+L_\tau(C(M_0))) \| \psihn{n} \|_{H^2} + C_\text{st}(\tau, M_0)\tau^2 \notag \\
		%			&\quad + \tau^2 |\ln(\tau)| + \tau h^2 |\ln(h)|, \qquad 0 \leq n \leq \frac{T}{\tau}-1.  
		%		\end{align}
	We prove the result by mathematical induction. Assume that {for all $n$ such that $0 \leq n \leq m \leq T/\tau-1$, we have}
	\begin{equation}\label{eq:assumption}
		\begin{aligned}
			&\| e^{n} \|_{L^2} \leq C_0( C_1 \tau |\ln(\tau)| |\ln(h)|+h^2|\ln(h)|), \quad \| \psihn{n} \|_{H^\frac{7}{4}} \leq 1+M, \\
			&\| \psihn{n} \|_{H^2} \leq C_1 |\ln(h)|, 
		\end{aligned}
	\end{equation}
	where $C_0 = 2e^{C_\text{s}T} C(M) $ and $C_1 = C_\text{inv} C_0 + 1$ are both fixed constants depending exclusively on $M$ and $T$. {We shall prove in the following that \cref{eq:assumption} holds for $n = m+1$.} From \cref{eq:error_eq}, using the assumptions \cref{eq:assumption}, we obtain,  
	\begin{align*}
		\| e^{n+1} \|_{L^2} 
		&\leq e^{C_\text{s}\tau} \| e^n \|_{L^2} + C(M) \tau^2 |\ln(\tau)| + C(M)  \tau h^2|\ln(h)| \\
		&\quad + C(M) C_1 \tau^2 |\ln (\tau)| |\ln(h)|, \qquad 0 \leq n \leq m, 
	\end{align*}
	which implies, by the discrete Gronwall's inequality, that 
	\begin{align}\label{eq:induction1}
		\| e^{m+1} \|_{L^2} 
		&\leq e^{C_\text{s}T} C(M) (\tau |\ln (\tau)|+  h^2|\ln(h)| + C_1 \tau |\ln(\tau)| |\ln (h)|) \notag \\
		&\leq 2e^{C_\text{s}T} C(M) (C_1\tau |\ln (\tau)| |\ln(h)| + h^2|\ln(h)|) \notag \\
		&=  C_0 (C_1\tau |\ln (\tau)| |\ln(h)| + h^2|\ln(h)|). 
	\end{align}
	Then by the inverse inequality \cref{eq:inv} with $\alpha = 7/4$, there exists $h_0$ sufficiently small depending exclusively on $M$ and $T$ such that when $h<h_0$, we have, by recalling \cref{eq:induction1} and $\tau |\ln \tau| \leq \tilde C h^2/|\ln h|$, 
	\begin{align}\label{eq:induction2}
		\| \psihn{m+1} \|_{H^\frac{7}{4}} 
		&\leq \| e^{m+1} \|_{H^\frac{7}{4}} + \| P_N \psi(t_{m+1}) \|_{H^\frac{7}{4}} \leq C_\text{inv}h^{-\frac{7}{4}} \| e^{m+1} \|_{L^2} + M \notag \\
		&\leq C_\text{inv}C_0 (C_1 \tilde C h^\frac{1}{4} + h^\frac{1}{4} |\ln(h)|) + M \leq 1+M. 
	\end{align}
	Moreover, using the inverse inequality \cref{eq:inv} with $\alpha = 2$, recalling $\tau|\ln(\tau)| \leq \tilde C h^2/|\ln(h)|$ and \cref{eq:induction1}, we have
	\begin{align}\label{eq:induction_H2}
		\| \psihn{m+1} \|_{H^2} 
		&\leq \| e^{m+1} \|_{H^2} + \| P_N \psi(t_{m+1}) \|_{H^2} \notag \\
		&\leq C_\text{inv}h^{-2} \| e^{m+1} \|_{L^2} + M \notag \\
		&\leq C_\text{inv}C_0\left(C_1 \tilde C + |\ln(h)|  \right) + M,   
	\end{align}
	which implies, by choosing $h_0$ such that $C_\text{inv} C_0 C_1 \tilde C + M \leq |\ln h_0| \leq |\ln h|$, 
	\begin{align}\label{eq:induction3}
		\| \psihn{m+1} \|_{H^2} 
		&\leq C_\text{inv}C_0 |\ln(h)| + C_\text{inv}C_0C_1 \tilde C + M \notag \\ 
		&\leq (C_\text{inv}C_0+1) |\ln(h)| = C_1 |\ln (h)|.  
	\end{align}
	Combing \cref{eq:induction1,eq:induction2,eq:induction3}, we prove \cref{eq:assumption} for $n = m+1$ and thus for all $0 \leq n \leq T/\tau $ by mathematical induction. By noting that $e^n \in X_N$ for $n \geq 0$, the $H^1$-norm error bound can be obtained from the $L^2$-norm error  bound directly with the inverse inequality under the time step size restriction $\tau |\ln \tau| \lesssim h^2/|\ln h|$ as
	\begin{align*}
		\| e^{n} \|_{H^1} \lesssim h^{-1} \| e^{n} \|_{L^2} 
		&\lesssim \sqrt{\frac{\tau |\ln \tau| |\ln h|}{h^2}} \sqrt{\tau|\ln \tau| |\ln h|} + h |\ln h| \notag \\
		&\leq \sqrt{\tau}|\ln \tau| + h |\ln h|.  
	\end{align*}
	The proof is thus completed. 
\end{proof}

\section{Numerical results}\label{sec:num}
In this section, we provide some numerical results to validate our error estimate for the EWI-FS method \cref{eq:EWI-FS_scheme} and to show the necessity of the time step size restriction. 
%We also present comparisons with the time-splitting method to demonstrate the superiority of the EWI-FS method. 
We also apply our method to study the soliton collision under disorder potential in one dimension (1D) and vortex dipole dynamics in two dimensions (2D). 

\subsection{Convergence test}
We first test the convergence of the EWI-FS method. In this subsection, we consider a one dimensional setting with $d=1$ and $\Omega = (-16, 16)$. To quantify the error, we define the error functions as follows:
\begin{equation*}
	e_{L^2}(t_n): = \| \psi(t_n) - \psihn{n} \|_{L^2}, \quad e_{H^1}(t_n) = \| \psi(t_n) - \psihn{n} \|_{H^1}, \quad 0 \leq n \leq T/\tau. 
\end{equation*}

The following two types of initial data will be considered: 
\begin{itemize}
	\item[(i)] $H^2$-initial datum
	\begin{equation}\label{eq:ini1}
		\psi_0(x) = x|x|^{0.51}e^{-x^2/2}, \quad x \in \Omega.
	\end{equation}
	
	\item[(ii)] smooth initial datum
	\begin{equation}\label{eq:ini2}
		\psi_0(x) = c_1e^{-\frac{k_1(x-x_0)^2}{2} - ivx} + c_2e^{-\frac{k_2(x+x_0)^2}{2} + ivx}, \quad x \in \Omega, 
	\end{equation}
	where $x_0, v, k_1, k_2, c_1, c_2$ are some real constants, i.e., the initial data is the sum of two Gaussons at location $\pm x_0$ with velocity $\pm v$.  
\end{itemize} 
%	\begin{itemize}
	%		\item[(i)] $H^2$-initial datum without potential
	%		\begin{equation}
		%			V_1(x) \equiv 0, \quad \psi_0(x) = x|x|^{0.51}e^{-x^2/2}, \quad x \in \Omega. 
		%		\end{equation}
	%		
	%		\item [(ii)] Two solitons interacting under a square-well potential
	%		\begin{equation}
		%			V_2(x) = \left\{
		%			\begin{aligned}
			%				&-4, && x \in (-2, 2) \\
			%				&0, && \text{otherwise} 
			%			\end{aligned}
		%			\right., \  \psi_0(x) = e^{-\frac{(x-4)^2}{2} - 2ix} + e^{-\frac{(x+4)^2}{2} + 2ix}, \  x\in \Omega. 
		%		\end{equation}
	%	\end{itemize} 
The initial datum \cref{eq:ini1} is chosen as an odd function to ensure that the singularity of the logarithmic nonlinearity at the origin is revealed since the exact solution satisfies $\psi(0, t) \equiv 0$. The initial data in \cref{eq:ini2} is used to simulate the collision of two Gaussons. As we will show in the following, we can observe significantly different behaviour of the EWI-FS method for both types of initial data compared to the power-type nonlinearity in \citep{bao2023_EWI,bao2023_sEWI}. %due to the influence of the singularity of the logarithmic nonlinearity

We start with the convergence test and fix $\Omega = (-16, 16)$, $\lambda = -1$ in \cref{LogSE}. The final time is chosen as $T=1$. The ``exact" solutions are computed using the Strang time-splitting Fourier spectral method \citep{bao2023_improved,bao2019} with $\tau = \tau_\text{e} = 10^{-6}$ and $h = h_\text{e} = 2^{-9}$. When testing the temporal convergence, we compute the errors with varying $\tau$ from $10^{-5}$ to $10^{-1}$ for each $h = 2^{-k} \ (k=2, \cdots, 7)$. When testing the spatial convergence, we fix $\tau = \tau_\text{e}$ and show the errors computed with varying $h$ from $2^{-5}$ to $2^{-1}$. 
%	As we will show in the following, the EWI-FS method exhibits behavior significantly different from that observed in cases of power-type nonlinearity \citep{bao2023_EWI,bao2023_sEWI}. Similar phenomena, though less clear than in case (i), can also be observed in case (ii). Moreover, although the initial datum $\psi_0$ in \cref{eq:ini2} is smooth, due to the low regularity of the discontinuous $L^\infty$-potential $V_2$, the exact solution is still of low regularity (roughly $H^{2.5}$). The exact solution is computed using the Strang time-splitting Fourier spectral method \citep{bao2023_improved,bao2019} with $\tau = \tau_\text{e} = 10^{-6}$ and $h = h_\text{e} = 2^{-9}$. 

We first consider the case of \cref{LogSE} with an $H^2$-initial datum \cref{eq:ini1} and without potential (i.e., $V(x) \equiv 0$). The temporal errors of the EWI-FS method in $L^2$- and $H^1$-norms are presented in \cref{fig:conv_dt_H2_ini_L2,fig:conv_dt_H2_ini_H1}. In the left figures, each line corresponds to the errors computed with a fixed mesh size $h$ and varying time step size $\tau$. In the right figures, each line represents the errors computed for $\tau$ and $h$ satisfying a fixed ratio $\tau = ch^2$. The spatial errors in $L^2$- and $H^1$-norms are shown in \cref{fig:conv_h} (a). 

From \cref{fig:conv_dt_H2_ini_L2,fig:conv_dt_H2_ini_H1}, we see that the EWI-FS method is first-order convergent in $L^2$-norm and half-order convergent in $H^1$-norm under an $H^2$-initial data and the time step size restriction $\tau \lesssim h^2$. Notably, when the time step size restriction is violated, i.e., in the regime where $\tau \gg h^2$, significant convergence order reduction can be observed in both $L^2$- and $H^1$-norms. Moreover, the optimal convergence orders are observed when $\tau = ch^2$ for all different $c$; however, one shall choose $c$ suitably small in practice to enter the asymptotical regime with reasonably small time steps. In terms of space, \cref{fig:conv_h} (a) shows that the spatial convergence is of second order in $L^2$-norm and first order in $H^1$-norm. These results confirm our error estimates in \cref{thm:main} and also indicate that the time step size restriction in \cref{thm:main} is necessary. 

\begin{figure}[htbp]
	\centering
	{\includegraphics[width=0.475\textwidth]{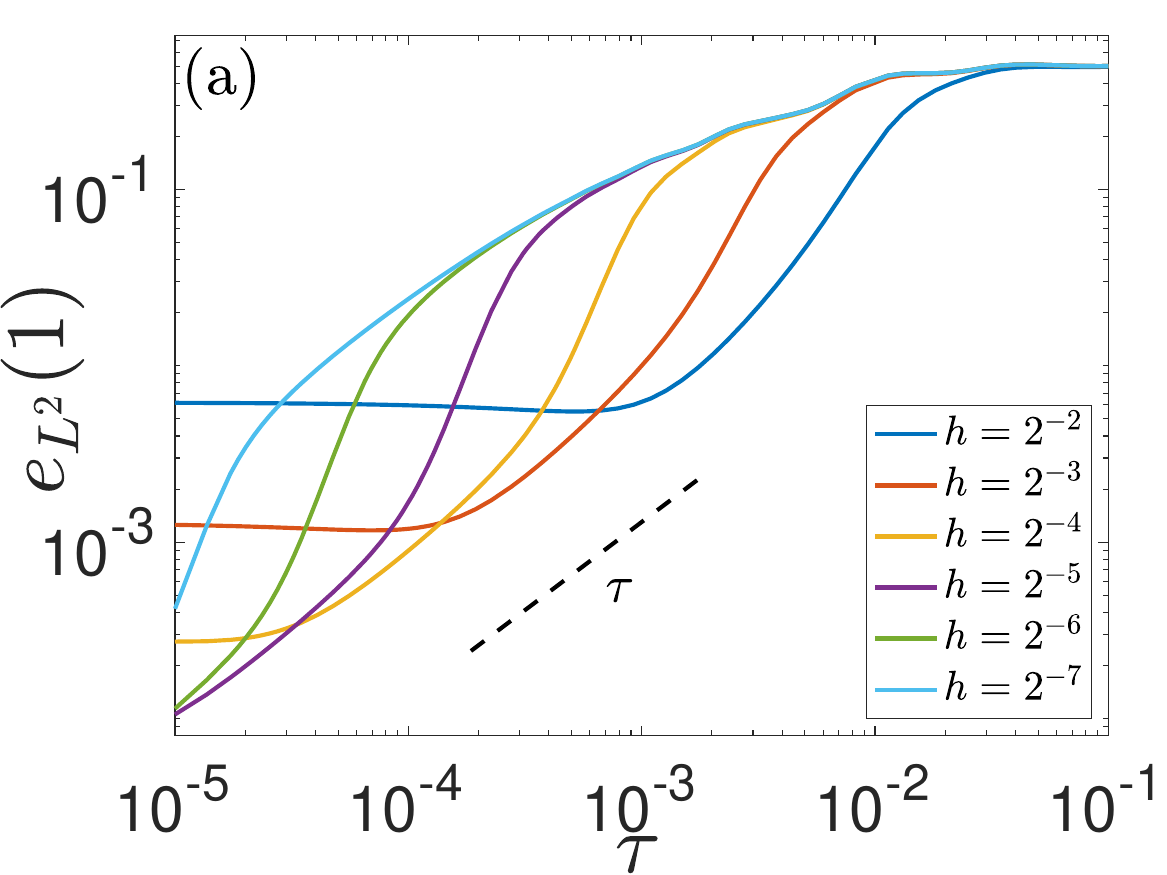}}\hspace{1em}
	{\includegraphics[width=0.475\textwidth]{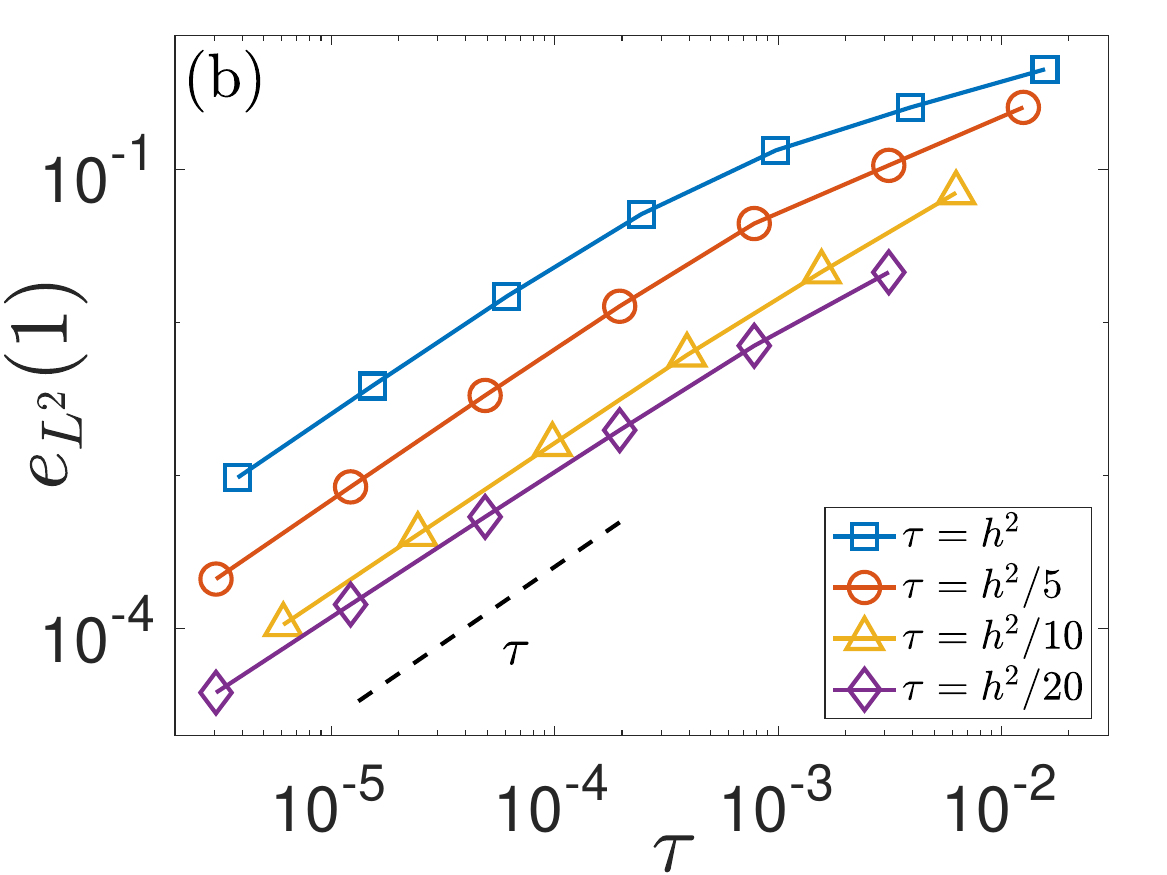}}
	\caption{Errors in $L^2$-norm of the EWI-FS method for the LogSE with $V(x) \equiv 0$ and the initial datum \cref{eq:ini1}: (a) varying $\tau$ for each fixed $h$ and (b) $\tau$ and $h$ satisfying $\tau/h^2 \equiv c$ with different ratio $c$}
	\label{fig:conv_dt_H2_ini_L2}
\end{figure}

\begin{figure}[htbp]
	\centering
	{\includegraphics[width=0.475\textwidth]{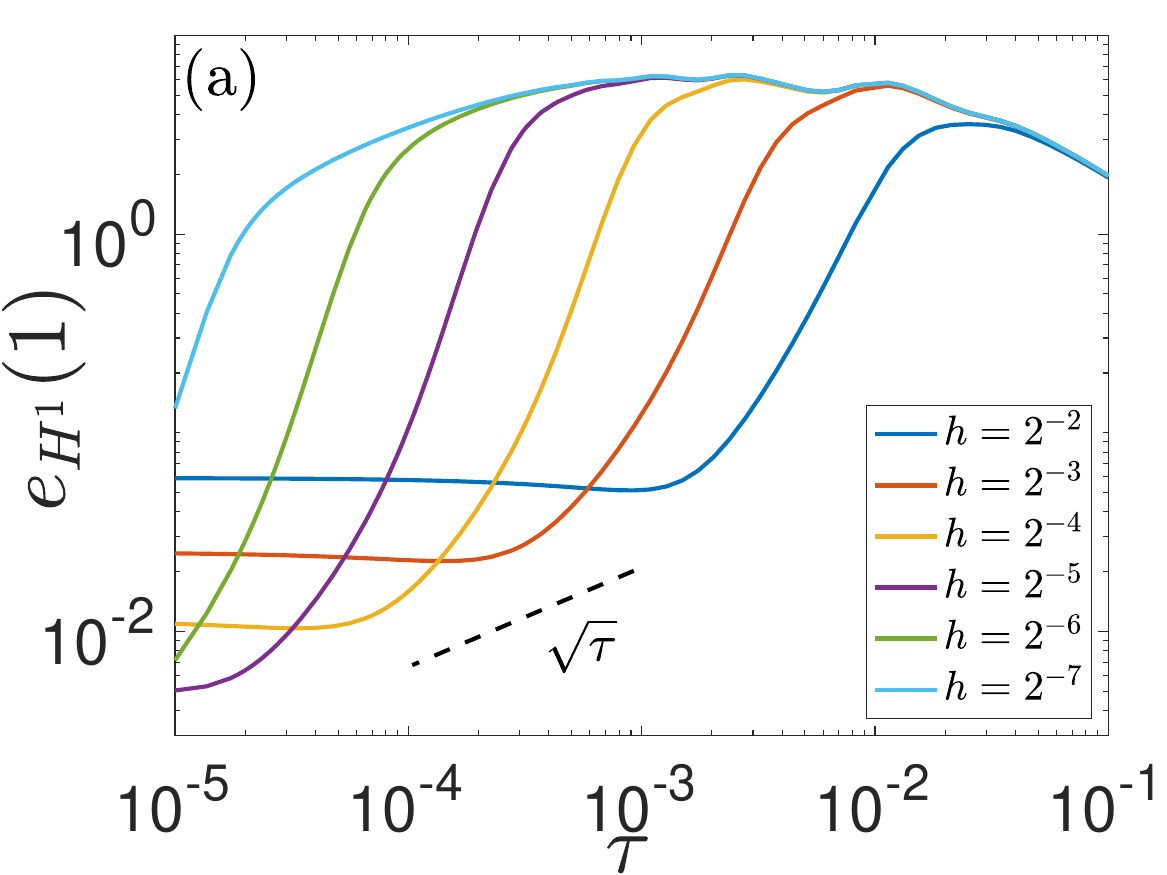}}\hspace{1em}
	{\includegraphics[width=0.475\textwidth]{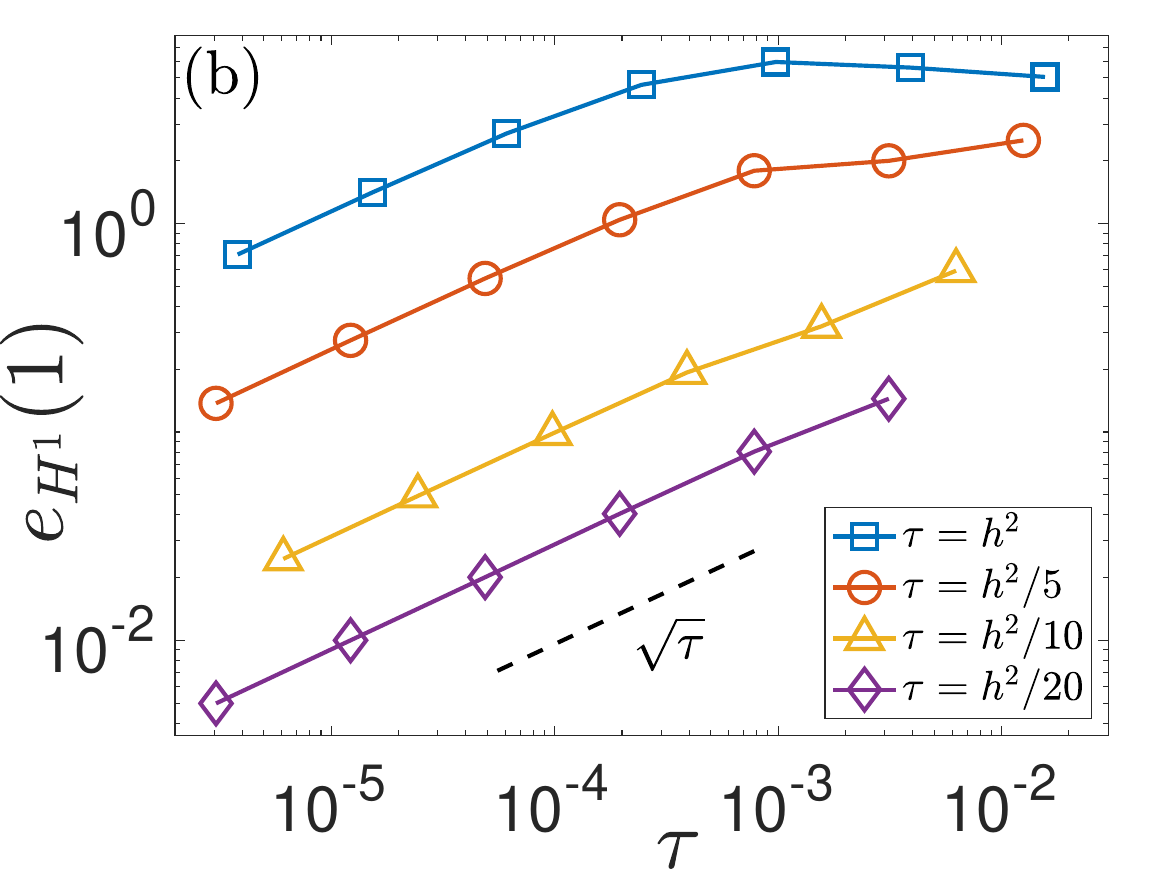}}
	\caption{Errors in $H^1$-norm of the EWI-FS method for the LogSE with $V(x) \equiv 0$ and the initial datum \cref{eq:ini1}: (a) varying $\tau$ for each fixed $h$ and (b) $\tau$ and $h$ satisfying $\tau/h^2 \equiv c$ with different ratio $c$}
	\label{fig:conv_dt_H2_ini_H1}
\end{figure}

%	\begin{table}[h!]
	%		\renewcommand{\arraystretch}{1.5}
	%		\centering
	%		\begin{tabular}{c|ccccc}
		%			\thickhline
		%			& $\tau, h $ & $\tau/4, \  h/2$ & $\tau/4^2, \  h/2^2$ & $\tau/4^3, \  h/2^3$ & $\tau/4^4, \  h/2^4$ \\ \thickhline
		%			$e_{L^2}(1)$ &6.55E-2	&1.99E-2	&5.65E-3	&1.47E-3	&3.81E-4  \\ \hline
		%			order &- & 0.86 & 0.91 & 0.97 & 0.97 \\ \thickhline
		%			$e_{H^1}(1)$ &5.47E-1	&2.94E-1	&1.74E-1	&8.83E-2	&4.43E-2  \\ \hline
		%			order & - & 0.45 & 0.38 & 0.49 & 0.50 \\ \thickhline
		%		\end{tabular}
	%		\caption{Errors in $L^2$- and $H^1$-norms of the EWI-FS method with $(\tau, h)$-pairs satisfying $\tau \sim h^2$ highlighted in \cref{fig:conv_dt_H2_ini}}
	%		\label{tab:conv_dt_H2}
	%	\end{table}

\begin{figure}[htbp]
	\centering
	{\includegraphics[width=0.475\textwidth]{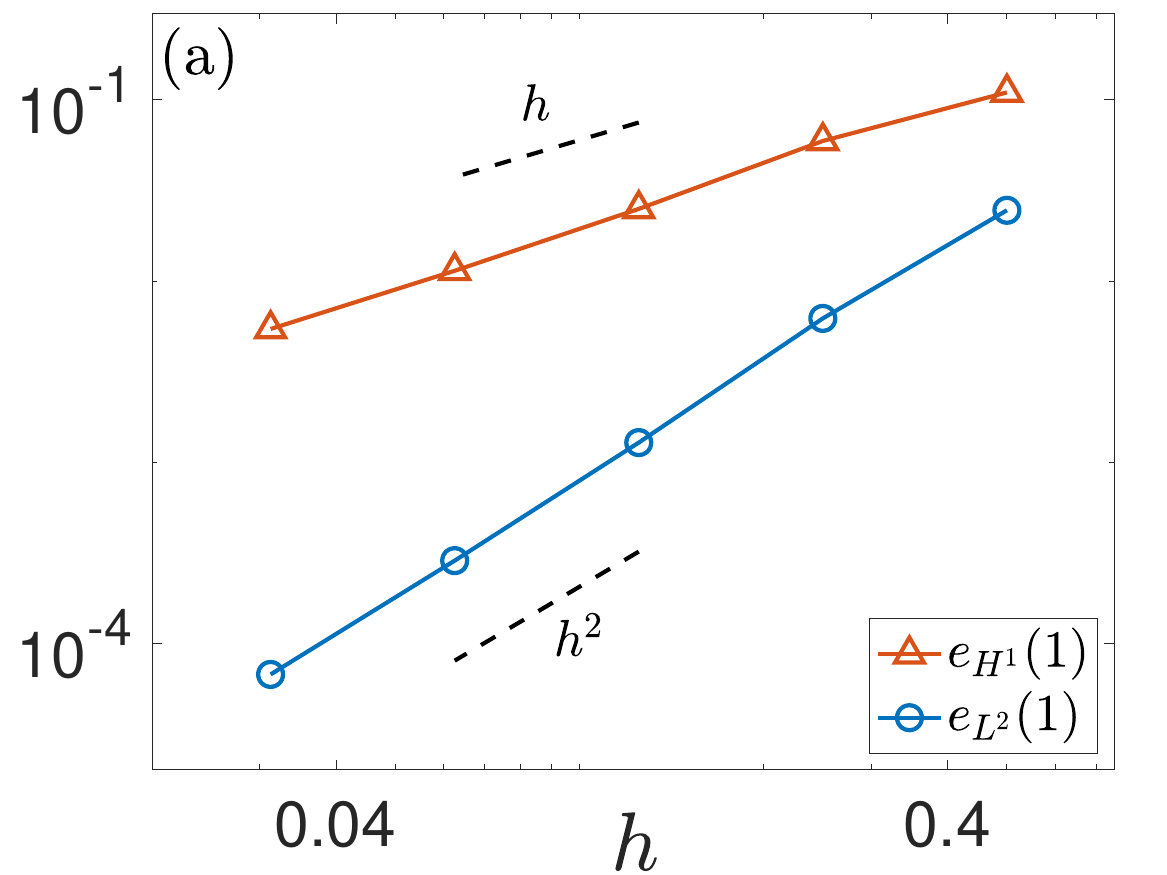}}\hspace{1em}
	{\includegraphics[width=0.475\textwidth]{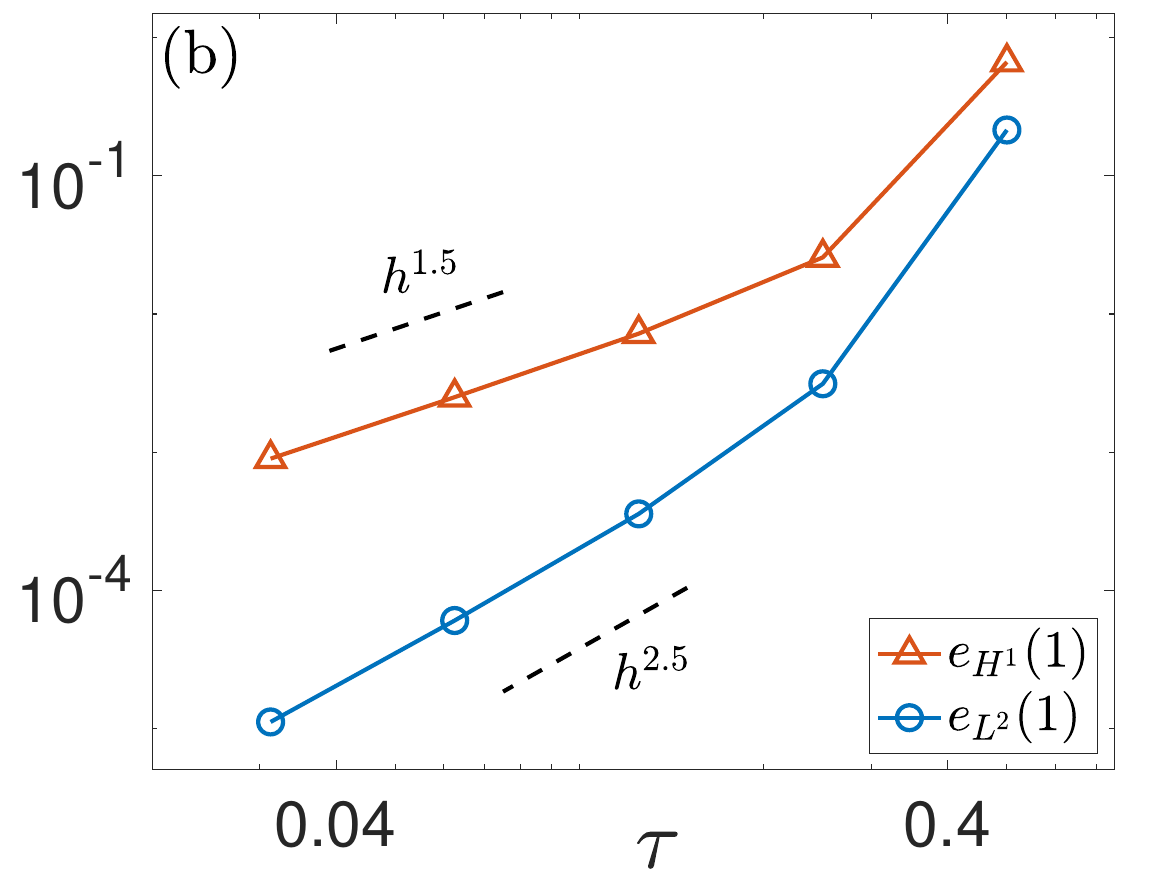}}
	\caption{Spatial errors in $L^2$- and $H^1$-norms of the EWI-FS method for the LogSE with (a) $V(x) \equiv 0$ and the initial datum \cref{eq:ini1} and (b) $V = V_\text{sw}$ and the initial datum \cref{eq:ini2}}
	\label{fig:conv_h}
\end{figure}

Then we study the LogSE \cref{LogSE} under the initial datum \cref{eq:ini2} with $x_0=4, v=2, c_1=c_2=k_1=k_2=1$, and a square-well potential $V=V_\text{sw}$ given by
\begin{equation}\label{eq:poten}
	V_\text{sw}(x) = \left\{
	\begin{aligned}
		&-4, && x \in (-2, 2) \\
		&0, && \text{otherwise} 
	\end{aligned}
	\right., \quad x\in \Omega. 
\end{equation}
In this case, although the initial datum \cref{eq:ini2} is smooth, due to the low regularity of the discontinuous potential $V_\text{sw}$ \cref{eq:poten}, the exact solution is still of low regularity (roughly $H^{2.5}$). The temporal errors (shown in \cref{fig:conv_dt_2Gaussian_ini_L2,fig:conv_dt_2Gaussian_ini_H1}) and the spatial errors (exhibited in \cref{fig:conv_h} (b)) are computed in the same way as described in the previous example. 

From \cref{fig:conv_dt_2Gaussian_ini_L2,fig:conv_dt_2Gaussian_ini_H1}, we see that the temporal error is of first-order in $L^2$-norm and $0.75$-order in $H^1$-norm under the time step size restriction $\tau \lesssim h^2$. Also, there is a similar convergence order reduction when $\tau \gg h^2$, though such order reduction in $L^2$-norm is not as severe as in the previous example. \cref{fig:conv_h} (b) demonstrates that the spatial error is of $2.5$ order in $L^2$-norm and $1.5$ order in $H^1$-norm, consistent with the $H^{2.5}$-regularity of the exact solution. These results validate our error estimates in \cref{thm:main,rem:highregularity}, and also suggest that the time step size restriction remains necessary in the simulation of Gaussons. In fact, our further numerical experiments (not shown here) indicate that the time step size is needed even for a single Gausson initial datum without potential, where the exact solution is known to be Gausson for all time. 

\begin{figure}[htbp]
	\centering
	{\includegraphics[width=0.475\textwidth]{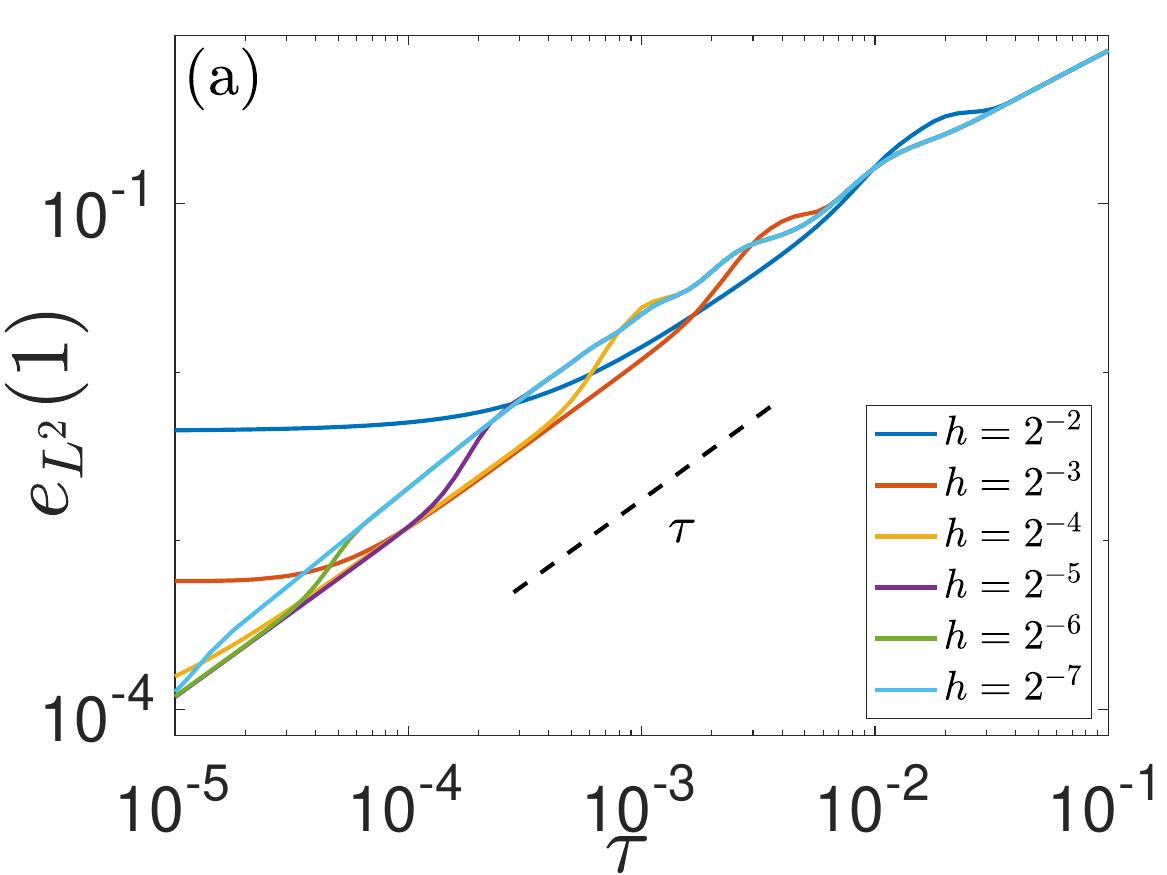}}\hspace{1em}
	{\includegraphics[width=0.475\textwidth]{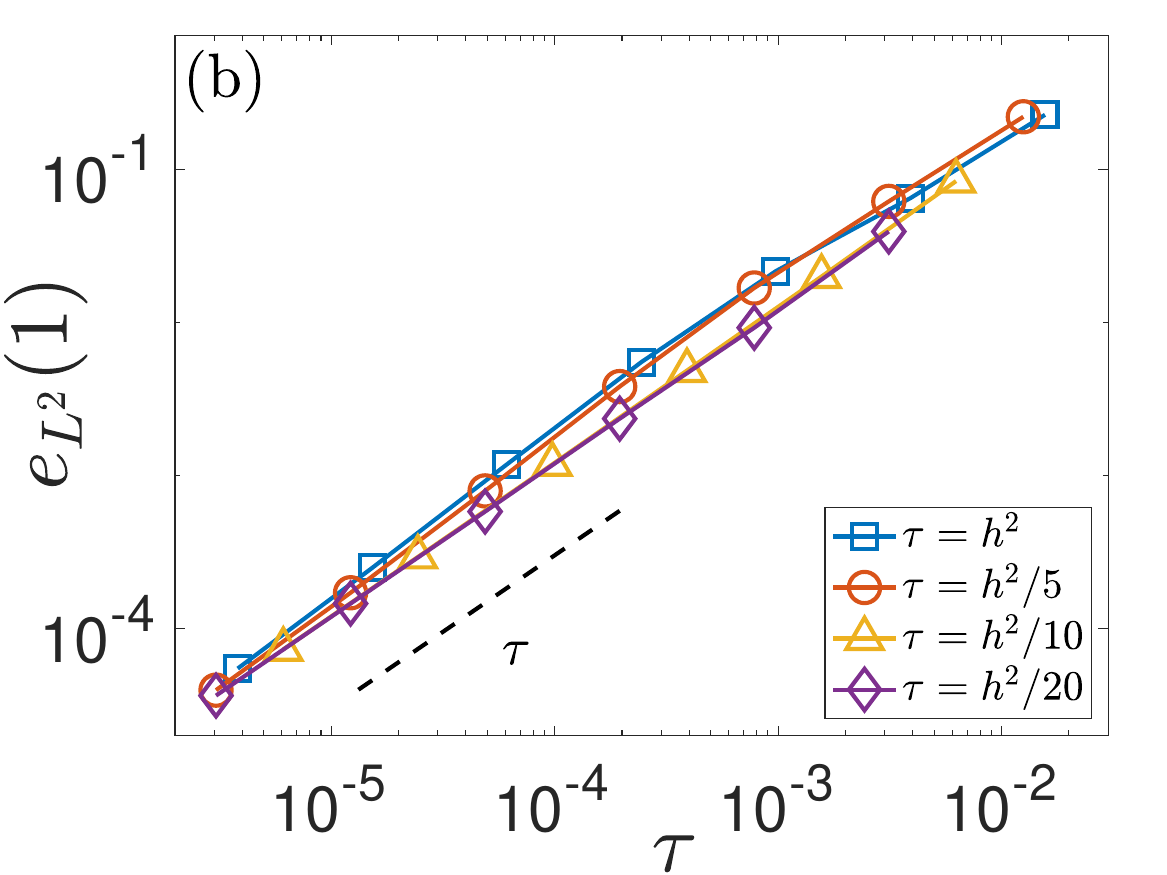}}
	\caption{Errors in $L^2$-norm of the EWI-FS method for the LogSE with $V = V_\text{sw}$ and the initial datum \cref{eq:ini2}: (a) varying $\tau$ for each fixed $h$ and (b) $\tau$ and $h$ satisfying $\tau/h^2 \equiv c$ with different ratio $c$}
	\label{fig:conv_dt_2Gaussian_ini_L2}
\end{figure}

\begin{figure}[htbp]
	\centering
	{\includegraphics[width=0.475\textwidth]{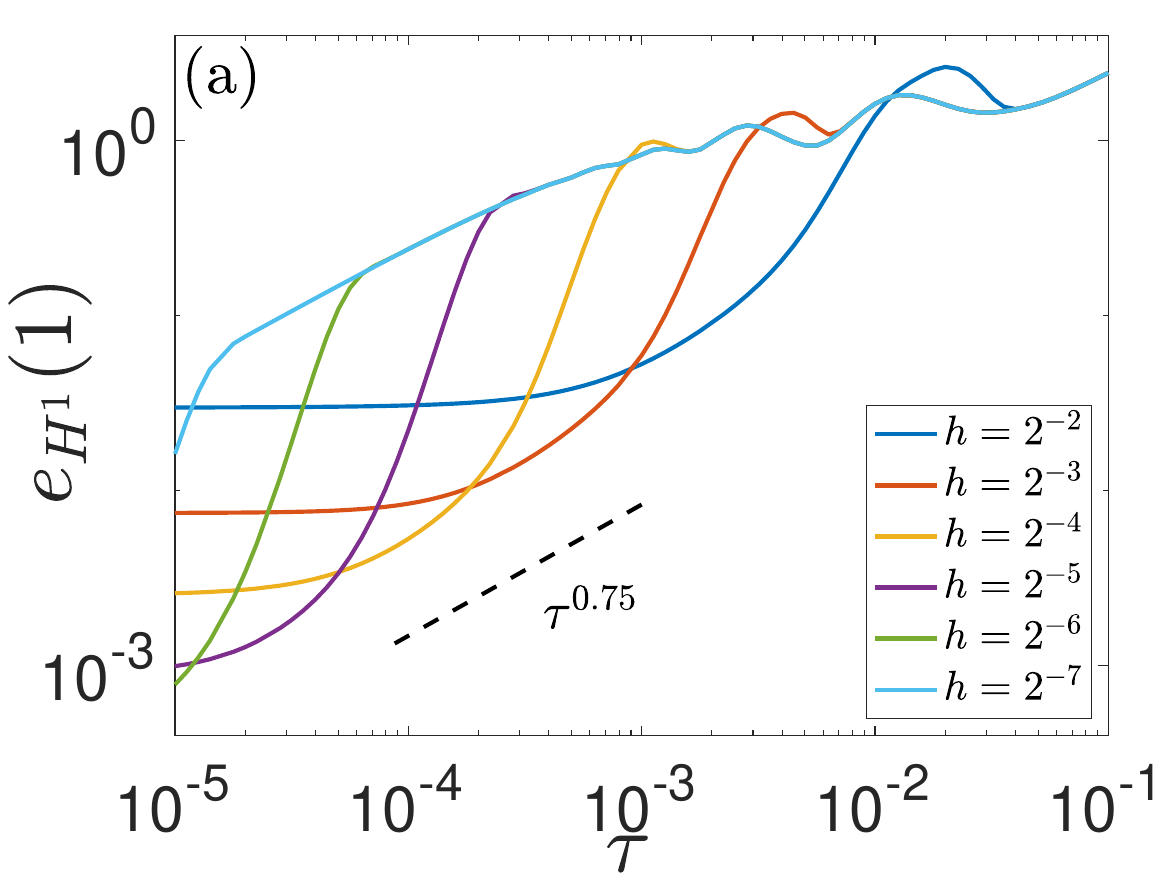}}\hspace{1em}
	{\includegraphics[width=0.475\textwidth]{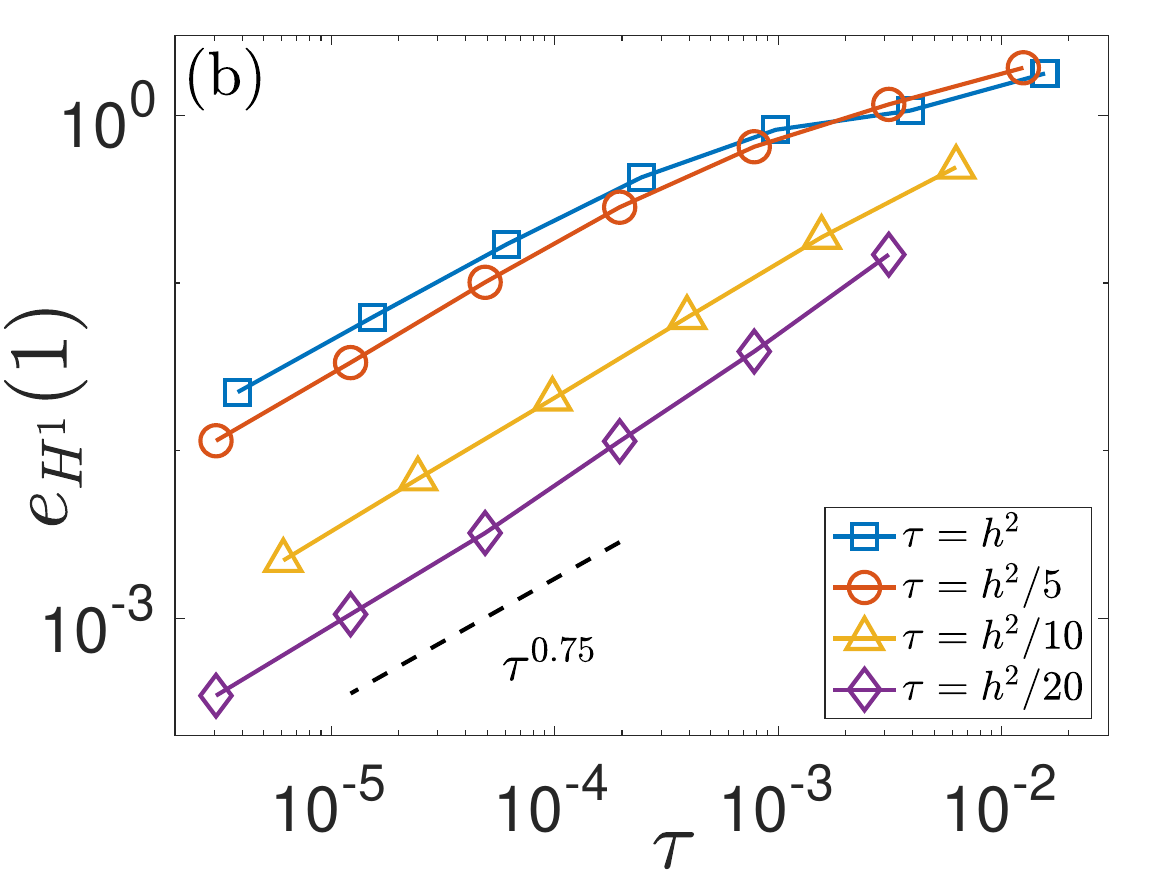}}
	\caption{Errors in $H^1$-norm of the EWI-FS method for the LogSE with $V = V_\text{sw}$ and the initial datum \cref{eq:ini2}: (a) varying $\tau$ for each fixed $h$ and (b) $\tau$ and $h$ satisfying $\tau/h^2 \equiv c$ with different ratio $c$}
	\label{fig:conv_dt_2Gaussian_ini_H1}
\end{figure}

{Finally, we present a numerical experiment to investigate the regularity threshold of the exact solution. We consider the LogSE \cref{LogSE} without potential (i.e. $V(x) \equiv 0$) and with a non-decaying initial datum $\psi_0(x) = \tanh(x)$ for $x \in \Omega$. In this example, we equip $\Omega$ with homogeneous Neumann boundary condition and use the cosine spectral method for spatial discretization as mentioned in \cref{rem:bc}. We test two cases with $\lambda = \pm 1$. The regularity of the numerical solution is estimated from the decay rate of its discrete cosine transform coefficients. From \cref{fig:regularity}, we see that the coefficients decay like $l^{-4}$ in both cases, suggesting that the exact solution is approximately in $H^{3.5}$. In fact, such $H^{3.5}$-regularity threshold is also observed in other numerical tests for odd initial data $\psi_0$ with $\psi_0'(0) \neq 0$. A heuristic explanation is that $x \ln x$ has $H^{1.5^-}$ regularity near the origin; hence by \cref{LogSE}, $\Delta \psi$ may inherit the same $H^{1.5^-}$ regularity, suggesting $\psi \in H^{3.5^-}$. Similar observation has been made for the NLSE \cref{NLSE} in \cite{bao2023_semi_smooth}. }

\begin{figure}[htbp]
	\centering
	{\includegraphics[width=0.475\textwidth]{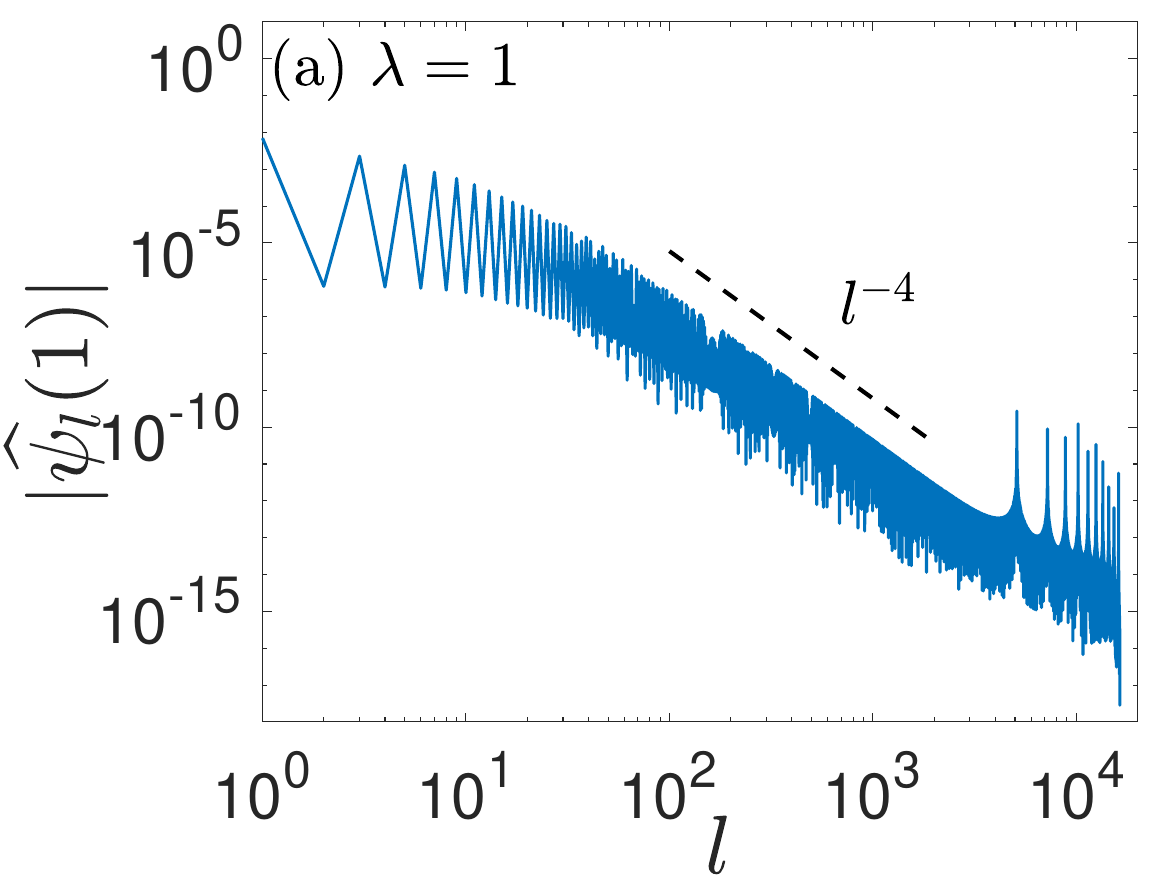}}\hspace{1em}
	{\includegraphics[width=0.475\textwidth]{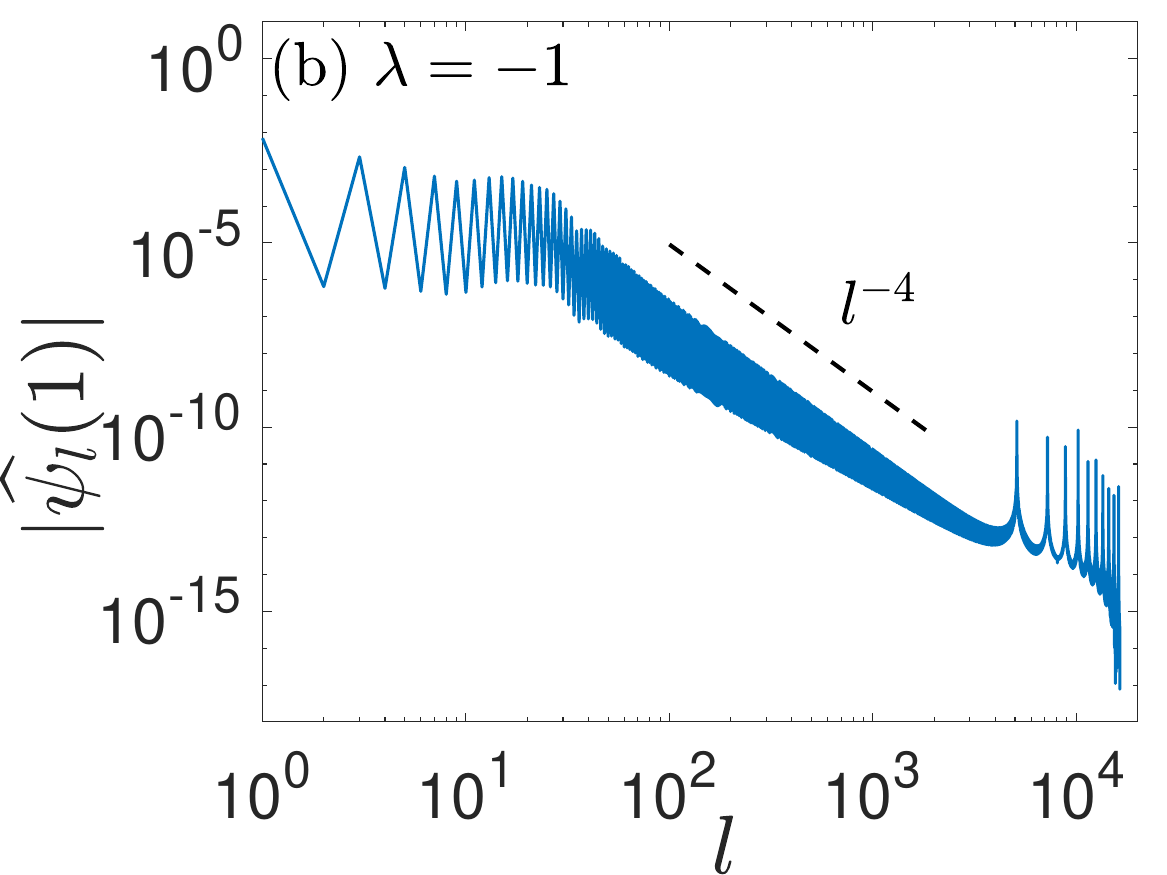}}
	\caption{Discrete cosine transform coefficients of the reference solution $ \psi^n \approx \psi(1) $ for the LogSE with $V = 0$ and the initial datum $\psi_0(x) = \tanh(x)$: (a) $\lambda = 1$ and (b) $\lambda = -1$}
	\label{fig:regularity}
\end{figure}

%	\begin{table}[h!]
	%		\renewcommand{\arraystretch}{1.5}
	%		\centering
	%		\begin{tabular}{c|ccccc}
		%			\thickhline
		%			& $\tau, h $ & $\tau/4, \  h/2$ & $\tau/4^2, \  h/2^2$ & $\tau/4^3, \  h/2^3$ & $\tau/4^4, \  h/2^4$ \\ \thickhline
		%			$e_{L^2}(1)$ &9.94E-2	&2.40E-2	&5.74E-3	&1.38E-3	&3.42E-4  \\ \hline
		%			order &- & 1.03 & 1.04 & 1.03 & 1.01 \\ \thickhline
		%			$e_{H^1}(1)$ &6.56E-01	&2.83E-01	&1.01E-01	&3.46E-02	&1.18E-02  \\ \hline
		%			order & - & 0.61 & 0.74 & 0.78 & 0.78 \\ \thickhline
		%		\end{tabular}
	%		\caption{Errors in $L^2$- and $H^1$-norms of the EWI-FS method with $(\tau, h)$-pairs satisfying $\tau \sim h^2$ highlighted in \cref{fig:conv_dt_2Gaussian_ini}}
	%		\label{tab:conv_dt_2Gaussian}
	%	\end{table}

\subsection{Application for soliton collision in 1D}
In this subsection, we apply the EWI-FS method to studying the soliton collision in a disorder medium characterized by a disorder potential $V = V_\text{d}$ given by
\begin{equation}\label{eq:poten2}
	V_\text{d}(x) = \text{real}\left(\sum_{l \in \mathcal{T}_{N_\text{Ref}}} (1+|\mu_l|^2)^{-\alpha/2-1/4} \xi_l e^{i \mu_l(x+L)} \right), \quad x \in \Omega = (-L, L), 
\end{equation}
where $\xi_l = \text{rand}(-1, 1) + i \ast \text{rand}(-1, 1)$ with $\text{rand}(-1, 1)$ returning a random number uniformly distributed in $(-1, 1)$, $N_\text{Ref} = 2^{18}$, and $\alpha > 0$ controlling the regularity of the potential. 

In the following simulation, we choose the computational domain $\Omega = (-32, 32)$ and the disorder potential $V_\text{d}$ \cref{eq:poten2} with $\alpha = 0$. The disorder potential $V_\text{d}$ used in the numerical simulation is shown in \cref{fig:plot_poten} which satisfies $\| V_\text{d} \|_{L^\infty} \approx 5.76$. The initial data is chosen as \cref{eq:ini2} with $x_0 = 4, k_1= k_2 = c_1 = c_2 =1$ and different velocity $v = 1, 2, 4, 8$. The numerical results are plotted in the right column of \cref{fig:dynamics}, where, for the comparison purpose, we also show the results in the absence of the disorder potential (i.e., $V(x) \equiv 0$) in the left column of \cref{fig:dynamics}. 

\begin{figure}[htbp]
	\centering
	%		{\includegraphics[width=0.475\textwidth]{figures/plot_poten_square-well}}\hspace{1em}
	{\includegraphics[width=1\textwidth]{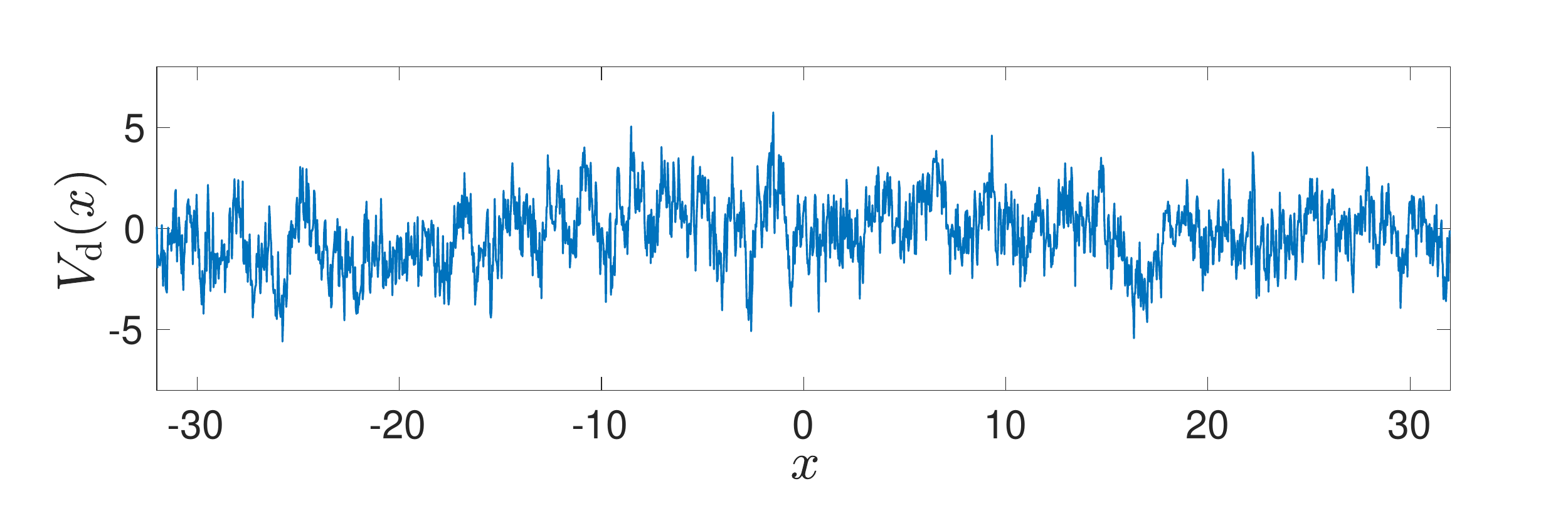}}
	\caption{An example of the disorder potential $V_\text{d}$}
	\label{fig:plot_poten}
\end{figure}

From \cref{fig:dynamics}, we observe that the effects of the disorder potential diminish as the velocity increases. In the lowest velocity case $v=1$, the two Gaussons are completely trapped in the potential. As the velocity increases, the trapping effect diminishes progressively and waves of smaller wavelength and amplitude are generated within the disorder potential. 

\begin{figure}[htbp]
	\centering
	{\includegraphics[width=0.425\textwidth]{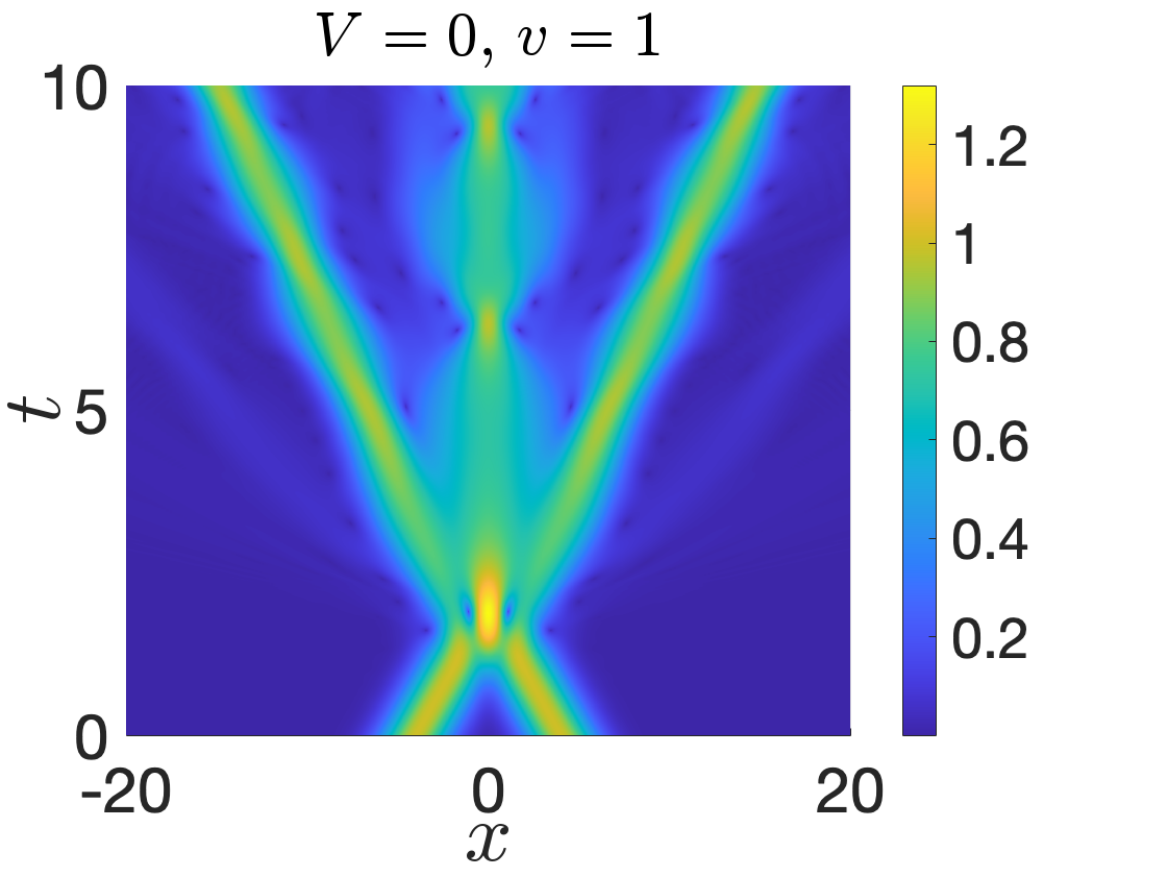}}\hspace{1em}
	{\includegraphics[width=0.425\textwidth]{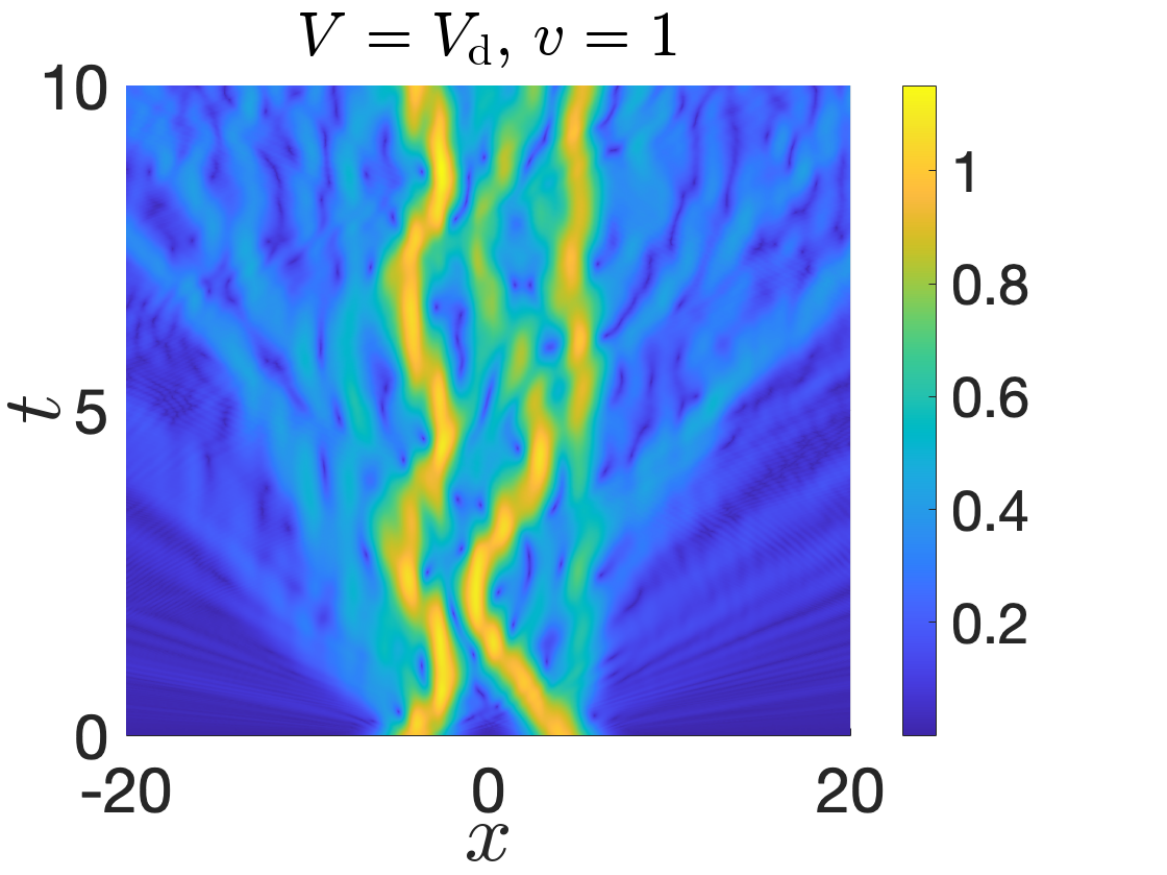}}\\
	{\includegraphics[width=0.425\textwidth]{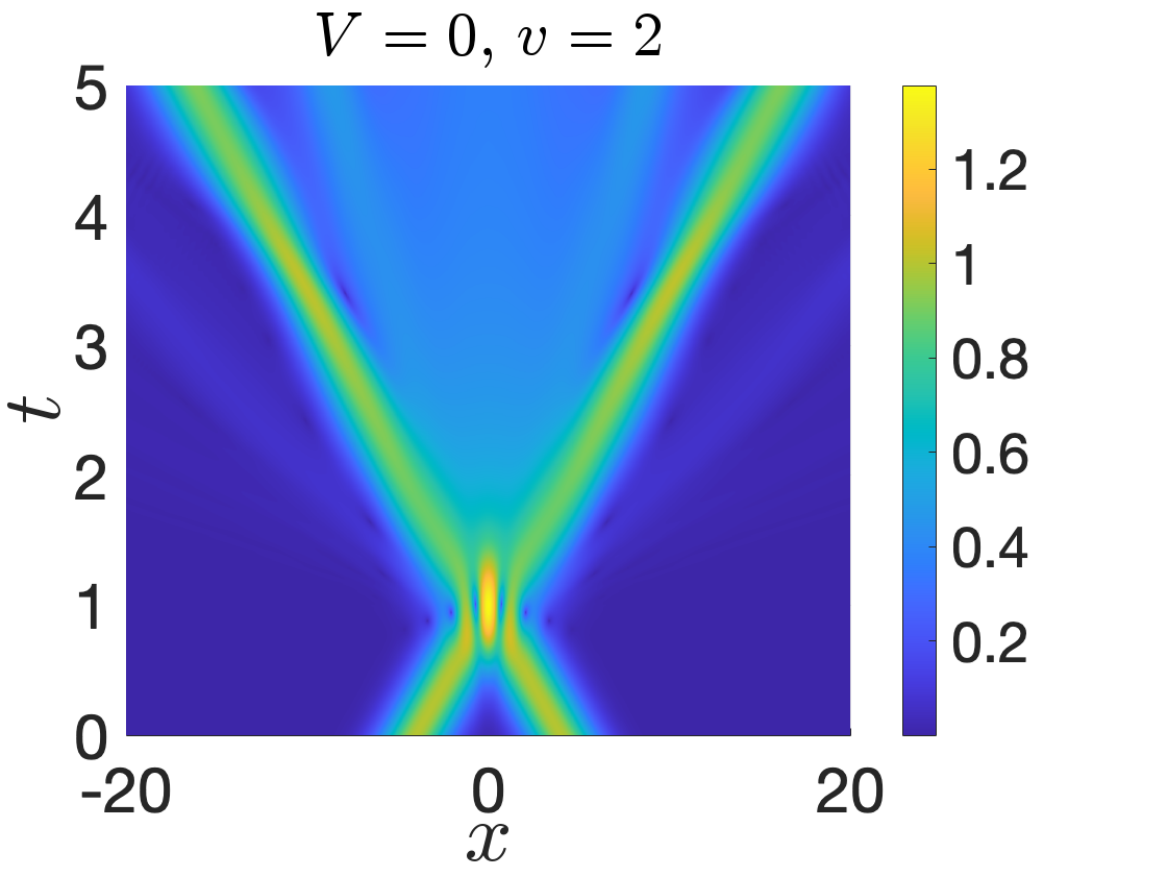}}\hspace{1em}
	{\includegraphics[width=0.425\textwidth]{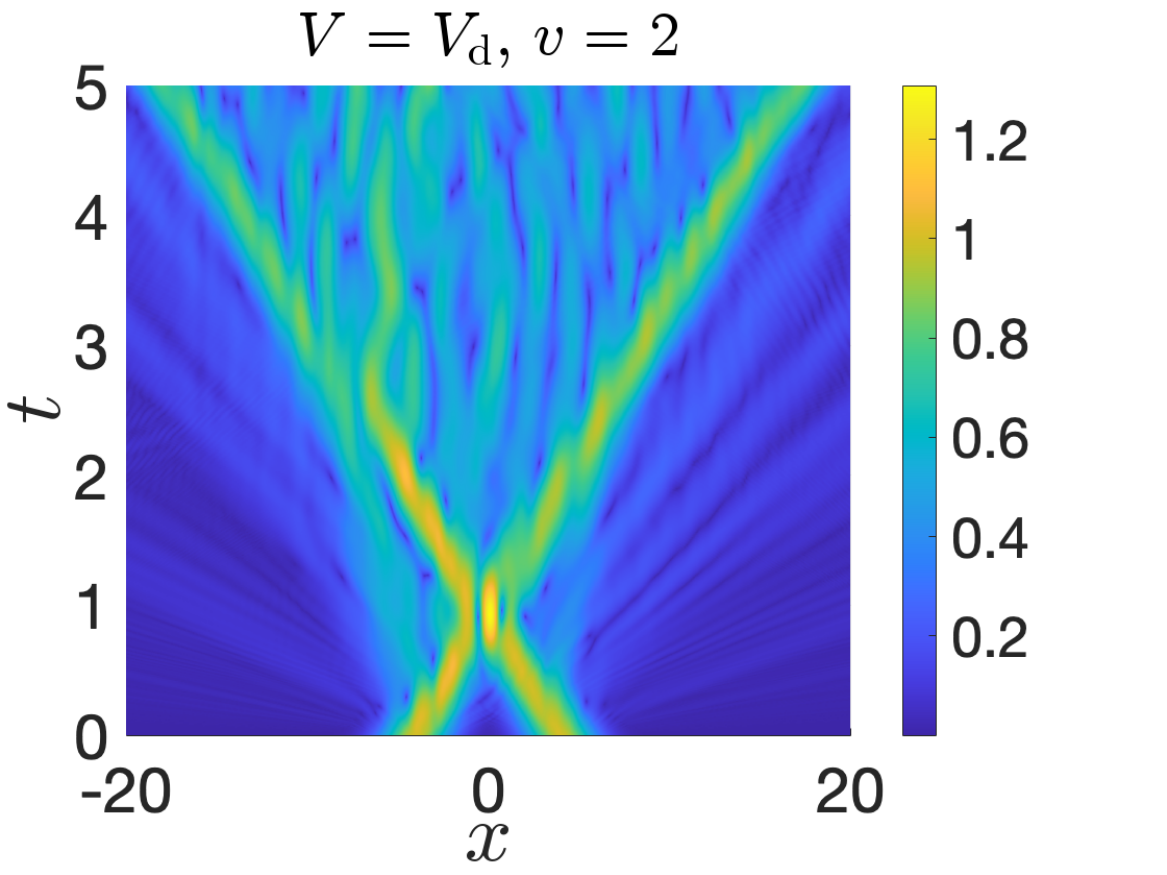}}\\
	{\includegraphics[width=0.425\textwidth]{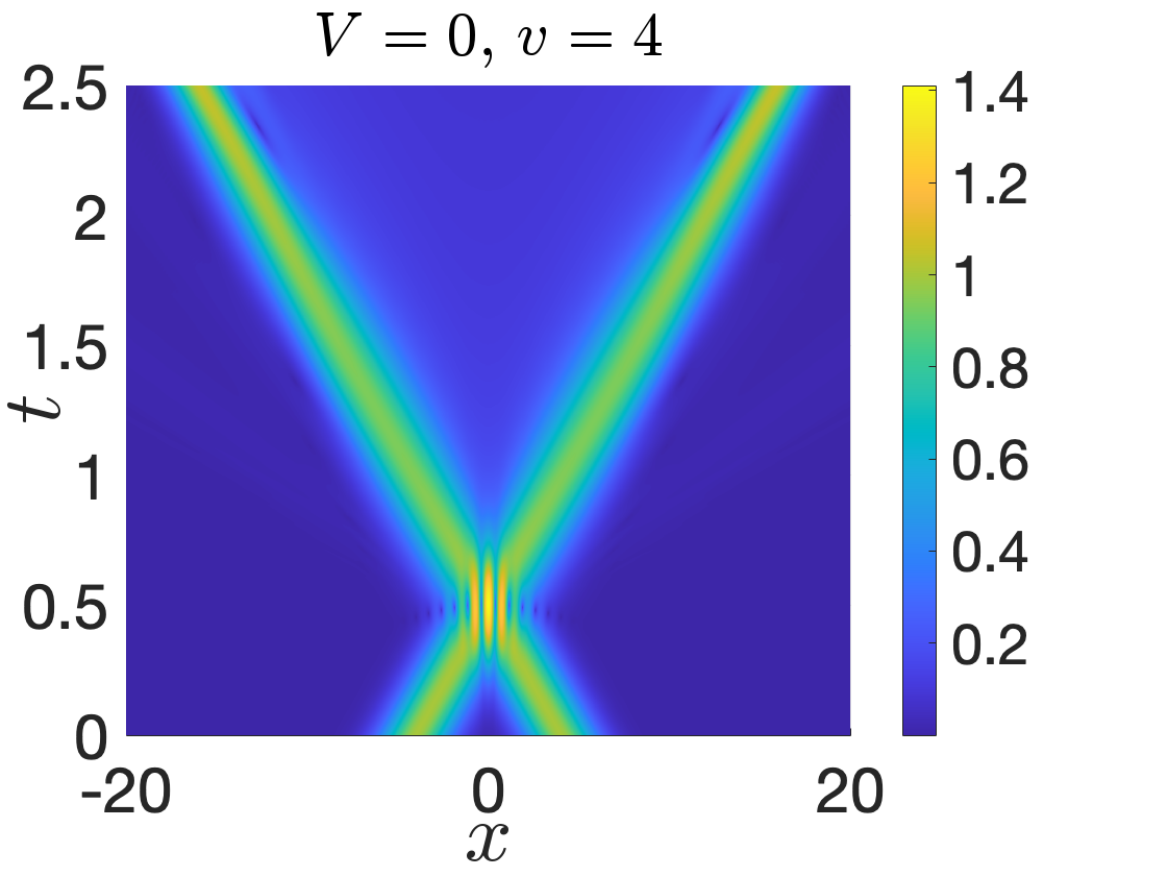}}\hspace{1em}
	{\includegraphics[width=0.425\textwidth]{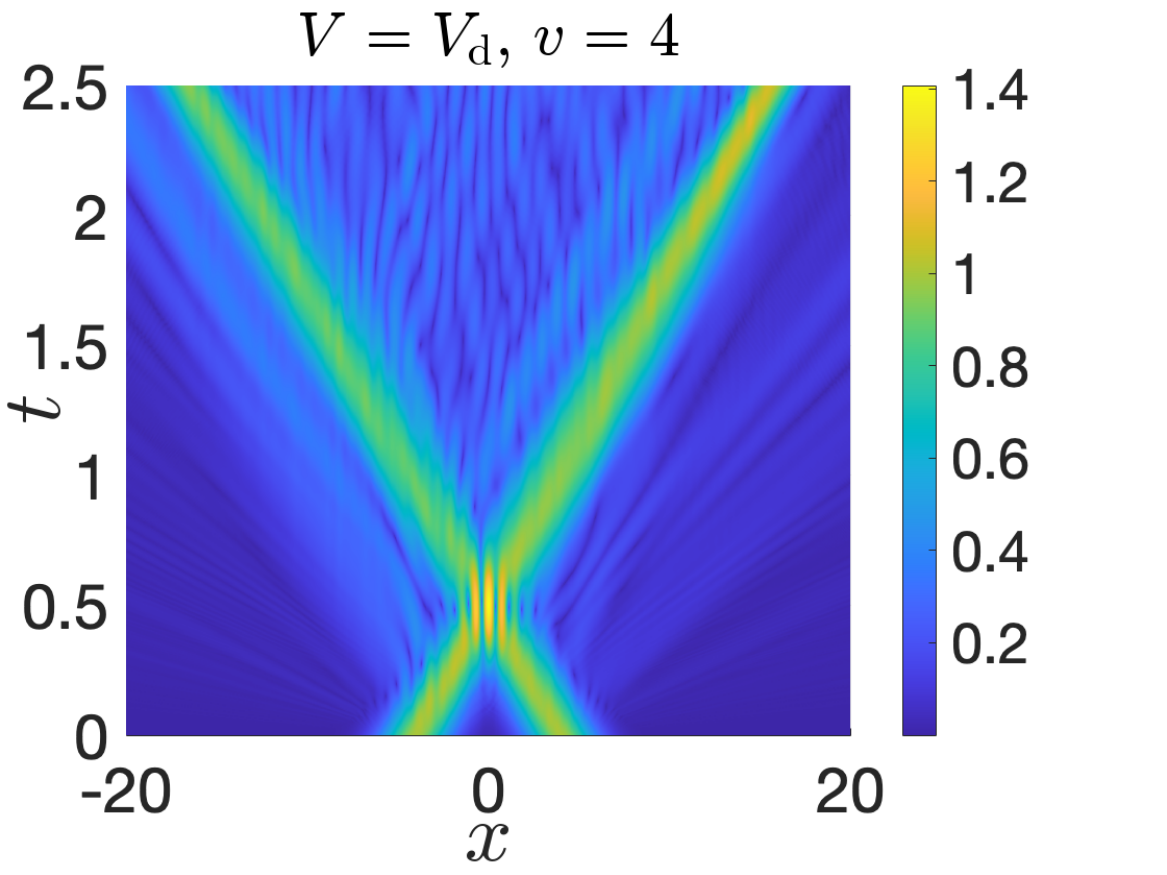}}\\
	{\includegraphics[width=0.425\textwidth]{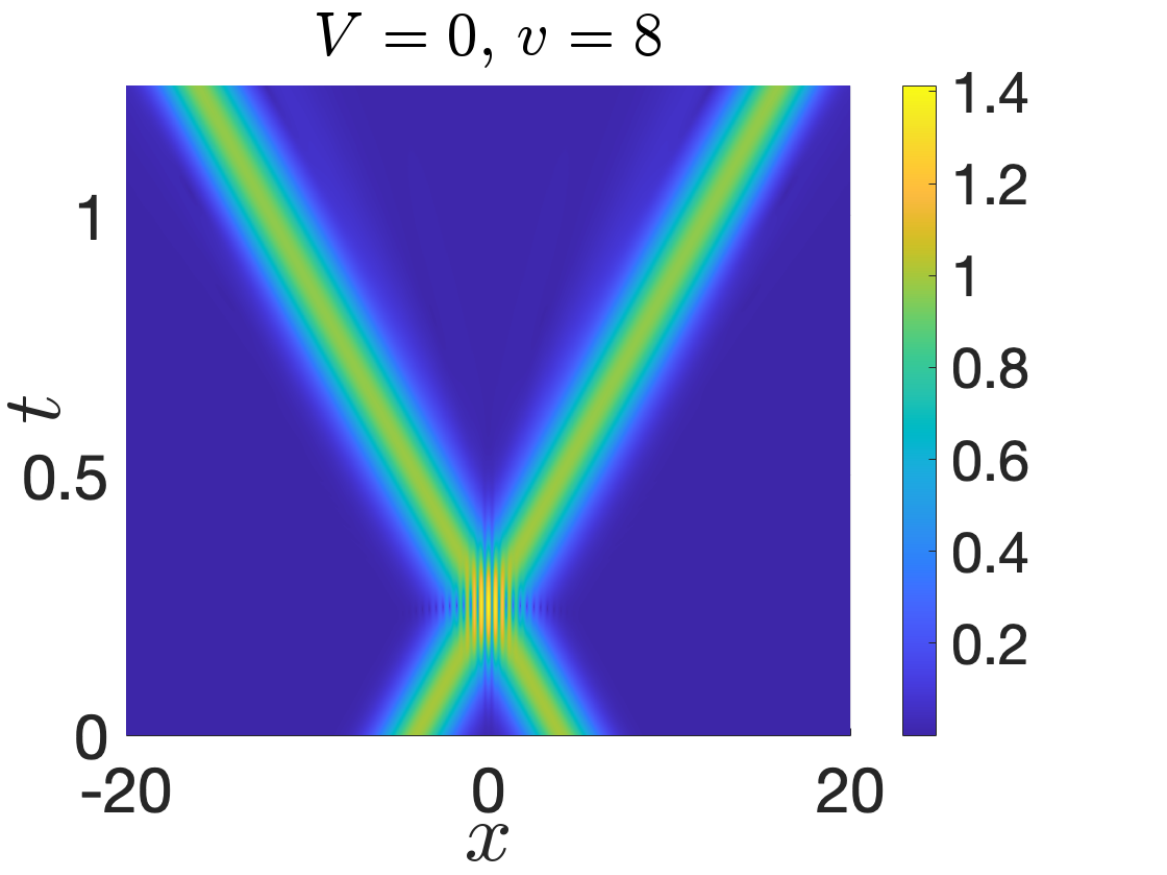}}\hspace{1em}
	{\includegraphics[width=0.425\textwidth]{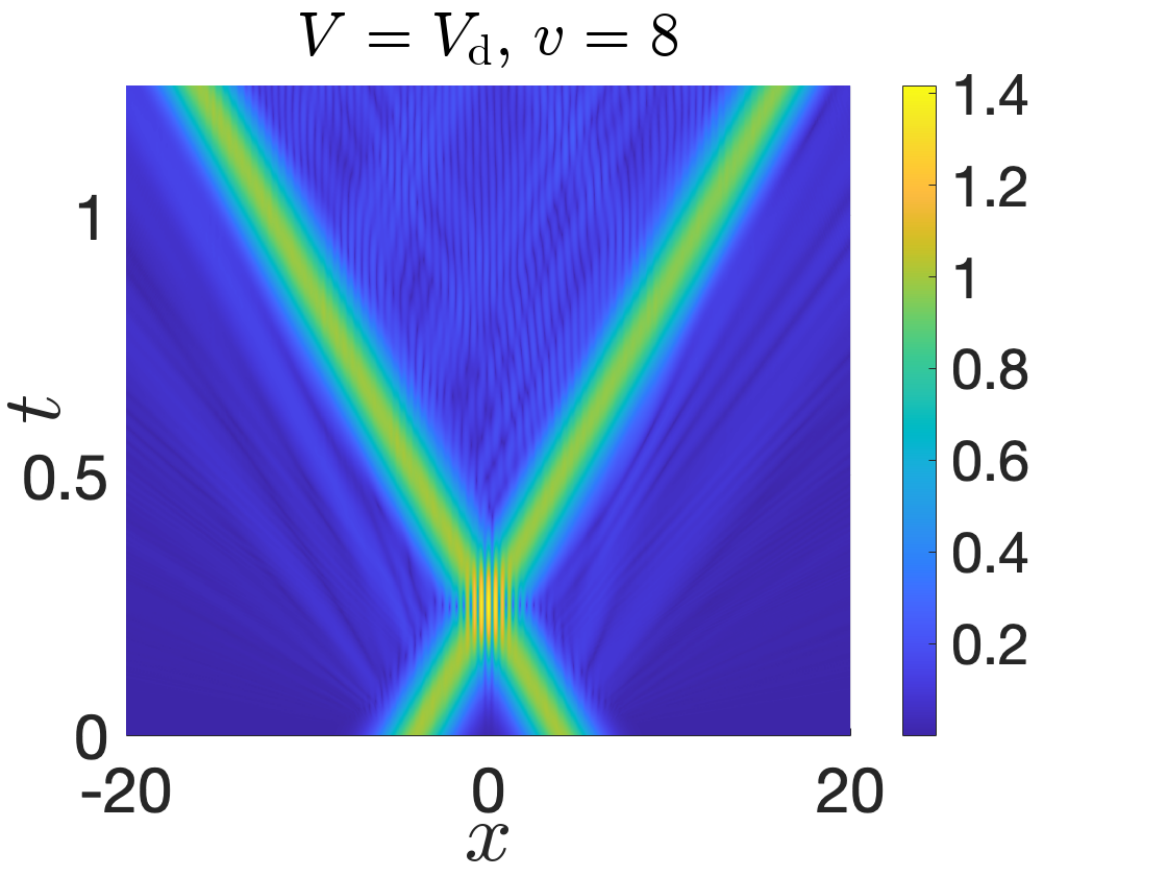}}\\
	\caption{Plots of $\sqrt{|\psi(x, t)|}$ for the LogSE under $V=0$ (left) and $V=V_\text{d}$ (right), and initial datum \cref{eq:ini2} with different velocity $v = 1, 2, 4, 8$ (from top to bottom)}
	\label{fig:dynamics}
\end{figure}

\subsection{Application for vortex dipole dynamics in 2D}
In this subsection, we apply the EWI-FS method \cref{eq:EWI-FS_scheme} to investigating the interactions of a vortex dipole under the LogSE \cref{LogSE} dynamics. Specifically, we consider a two-dimensional set-up with $d=2$, $\Omega=(-32, 32) \times (-32, 32)$ equipped with homogeneous Neumann boundary condition \citep{bao2014vortex}. In this example, we set $V(\vx) \equiv 0$ for $\vx \in \Omega$. The initial data is choose as
\begin{equation}\label{eq:ini_vt}
	\psi(\vx, 0) = \phi_+(x-x_0, y) \times \phi_-(x+x_0, y), \quad \vx = (x, y)^{\rm T} \in \Omega, 
\end{equation} 
where
\begin{equation}
	\phi_\pm(x, y) = u(|\vx|) e^{ \pm i \theta(\vx)}, \quad \vx = (x, y)^{\rm T} \in \Omega.  
\end{equation}
Here, $\theta(\vx) \in [0, 2\pi)$ is defined such that $x/|\vx| = \cos (\theta(\vx))$ and $y/|\vx| = \sin(\theta(\vx))$ for any $\mathbf{0} \neq \vx \in \R^2$, and $ u = u_\lambda $ is the solution of the following equation 
\begin{equation}\label{eq:vt_ini}
	-\frac{1}{r} \frac{\rmd}{\rmd r} \left(r\frac{\rmd}{\rmd r} \right) u(r) + \frac{u(r)}{r^2} + \lambda \ln(|u(r)|^2)u(r) = 0, \qquad 0<r<R_0, 
\end{equation}
with an inhomogeneous Dirichlet boundary condition
\begin{equation}
	u(0) = 0, \qquad u(R_0) = 1. 
\end{equation}
We choose $R_0 = 8$ and solve \cref{eq:vt_ini} numerically (see \cref{fig:f} for some examples of $u$). Moreover, we set $u(r) = 1$ for all $r \geq R_0$. 

\begin{figure}[htbp]
	\centering
	{\includegraphics[width=0.75\textwidth]{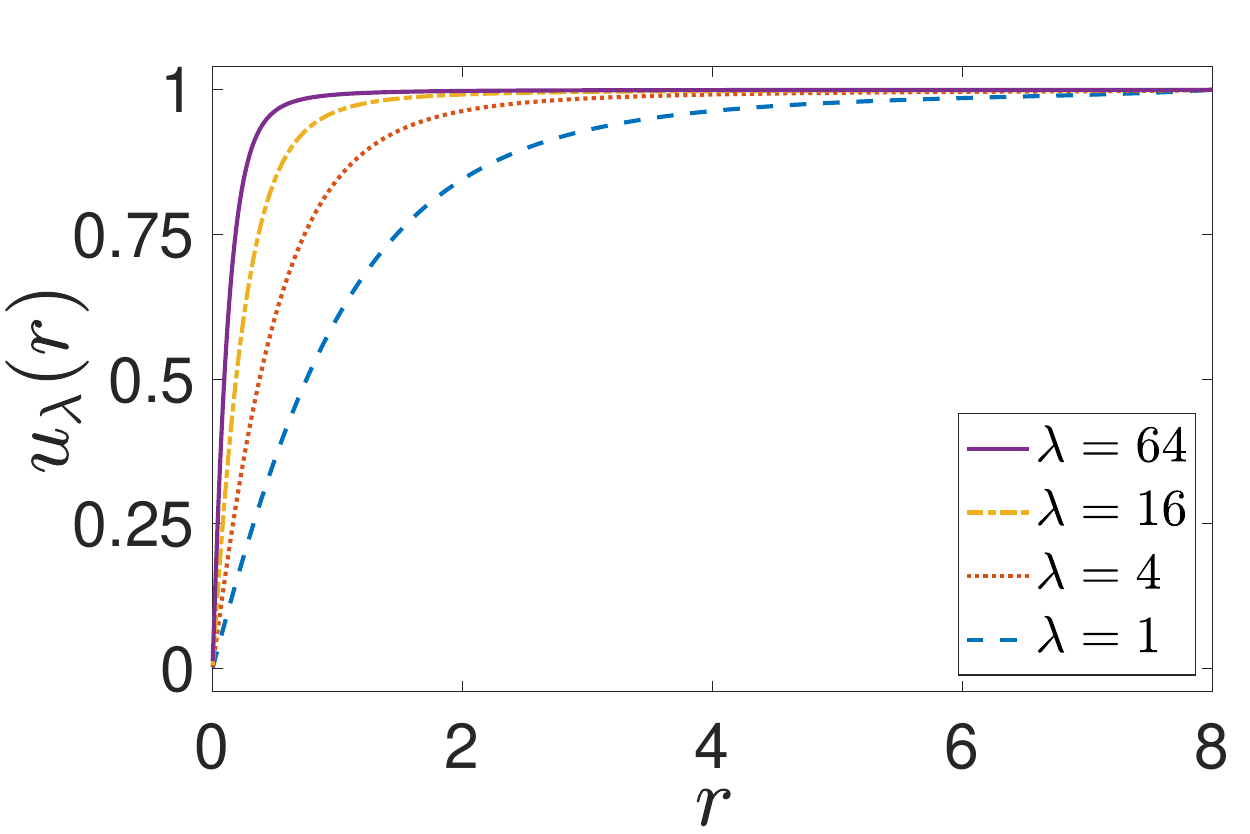}}
	\caption{Plots of solutions $u_\lambda(r)$ of \cref{eq:vt_ini} for different $\lambda = 1, 4, 16, 64$}
	\label{fig:f}
\end{figure}

We first study the influence of the distance between the vortices on the dynamics. To this end, we choose $\lambda = 16$ and consider {$x_0 = 0.5$ and $x_0 = 0.25$}. In computation, the mesh size in both directions are chosen as $h_x = h_y = 2^{-5}$ with the time step size $\tau = 10^{-5}$.  In \cref{fig:vortex}, we plot $|\psi(x, y, t)|$ at different time for both cases with the top row for $x_0 = 0.5$ and the bottom row for $x_0 = 0.25$. 

As illustrated in \cref{fig:vortex}, distinctly different dynamics emerge. For $x_0 = 0.5$, the two vortices move together in the 
$y$-direction while remaining stable and well-separated. In contrast, when the vortices are closer ($x_0 = 0.25$), they begin to merge and finally disappear.

\begin{figure}[htbp]
	\centering
	{\includegraphics[width=1\textwidth]{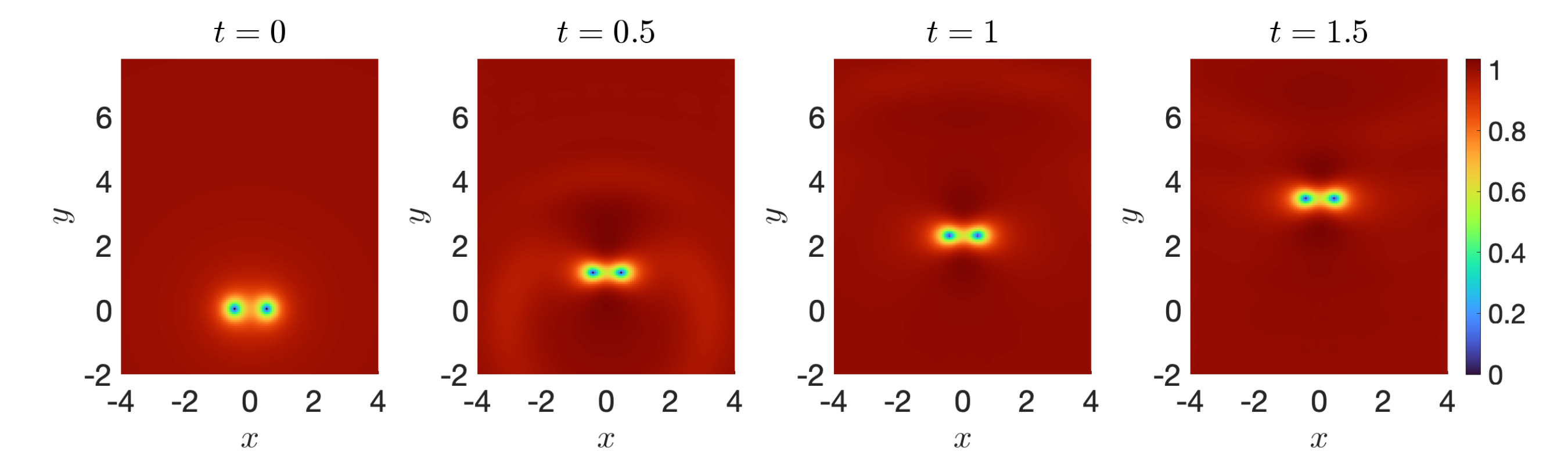}}\\
	{\includegraphics[width=1\textwidth]{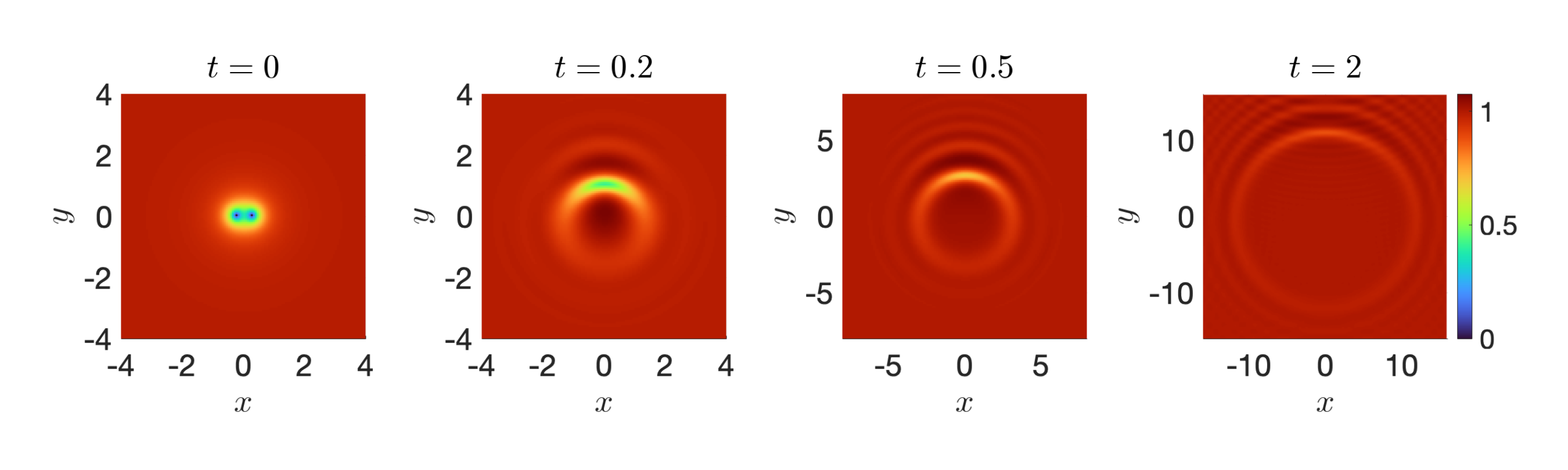}}
	\caption{Plots of $|\psi(x, y, t)|$ at different $t$ for the LogSE with $\lambda = 16$ and the initial datum \cref{eq:ini_vt} with $x_0 = 0.5$ (top) and $x_0 = 0.25$ (bottom). }
	\label{fig:vortex}
\end{figure}

Next, we examine the vortex dipole dynamics under the LogSE with $\lambda = -1$. \Cref{fig:vortex2} depicts $|\psi(x, y, t)|$ at various times under the initial datum \cref{eq:ini_vt} with $x_0 = 0.5$. In fact, in the focusing case, the dynamics for $x_0 = 0.5$ and $x_0$ are similar, and differ significantly from those in the defocusing case discussed earlier. The vortex core sizes expand rapidly, and the two vortices merge regardless of their initial distance. Finally, the two vortex form a single large vortex with small peaks inside and wave patterns radiating outward. 

\begin{figure}[htbp]
	\centering
	{\includegraphics[width=1\textwidth]{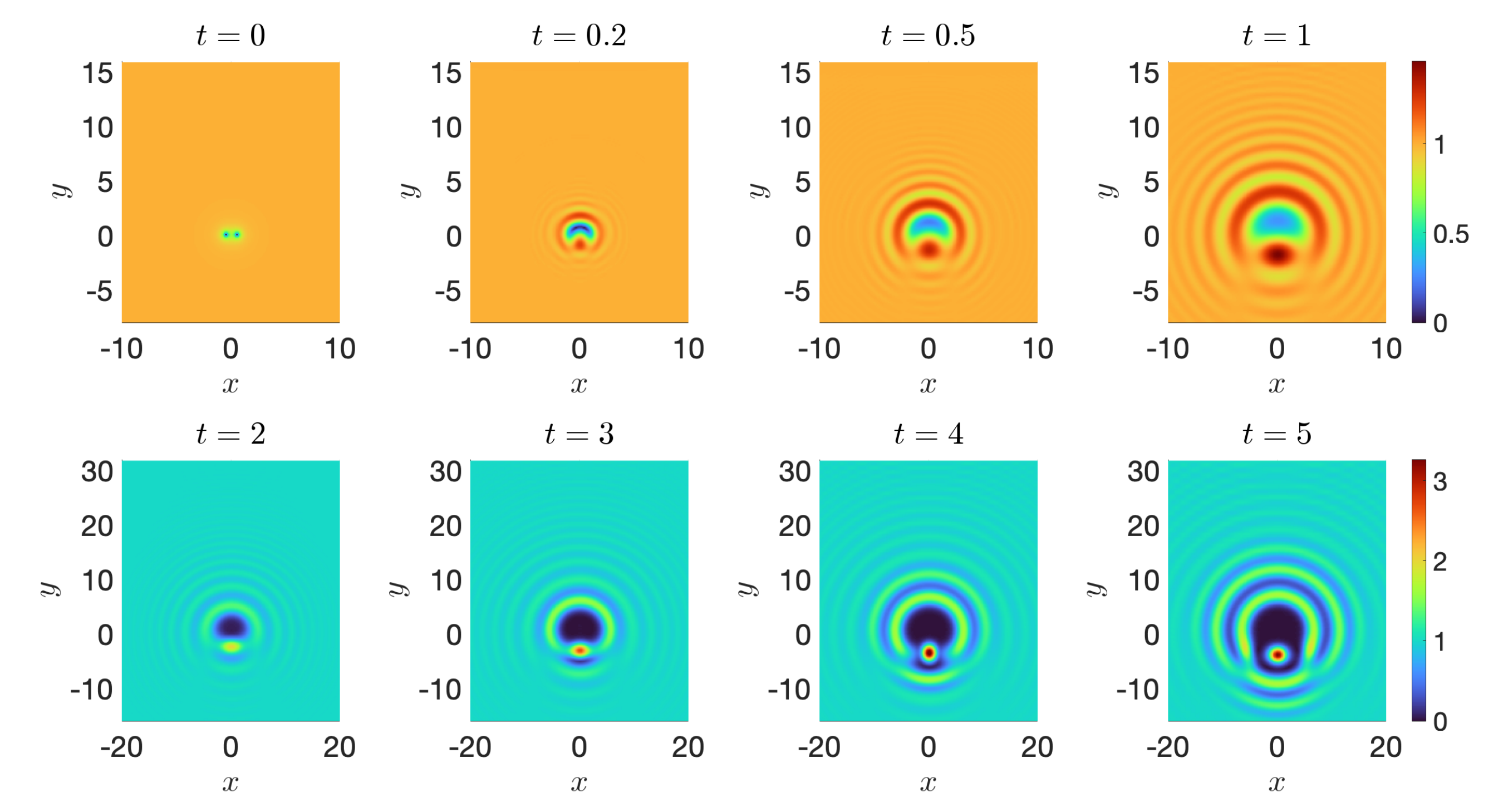}}\\
	%	{\includegraphics[width=1\textwidth]{figures/density_d025_beta-1}}
	\caption{Plots of $|\psi(x, y, t)|$ at different $t$ for the LogSE with $\lambda = -1$ and the initial datum \cref{eq:ini_vt} with $x_0 = 0.5$. }
	\label{fig:vortex2}
\end{figure}

\section{Conclusion}\label{sec:5}
We proposed and analyzed an EWI-FS method for the LogSE with low regularity potential. The EWI-FS method was proved to be, up to some logarithmic factors, first-order convergent in $L^2$-norm and half-order convergent in $H^1$-norm under the assumption of $H^2$-solution of the LogSE, which is theoretically guaranteed. Our analysis also revealed a CFL-type time step size restriction that is necessary in the implementation of the EWI-FS method to solve the LogSE. Extensive numerical results were reported to validate our error estimates and to show the necessity of the time step size restriction. Moreover, we have applied the EWI-FS to studying the soliton interaction in 1D and the vortex dipole dynamics in 2D. The validity of the current results for time-splitting methods will be investigated in our future work.

\section*{Acknowledgments}
The work of the first author was partially supported by the Ministry of Education of Singapore under its AcRF Tier 1 funding with WBS A-8003584-00-00. The work of the second author was partially supported by the National Natural Science Foundation of China (Grant No. 12401507) and the Science and Technology Project of Beijing Municipal Education Commission (Grant No. KM202410005013). The work of the third author was partially supported by the Ministry of Education of Singapore under its AcRF Tier 2 funding MOE-T2EP20222-0001 (A-8001562-00-00).

%\section*{Funding}

%\bibliographystyle{plain}
%\bibliography{mybib}

%USE THE BELOW OPTIONS IN CASE YOU NEED AUTHOR YEAR FORMAT.
\bibliographystyle{abbrvnat}
%\bibliography{reference}

\end{document}